\documentclass[12pt]{iopart}
\usepackage{subfig}

\usepackage{graphicx,color}
\usepackage{iopams}
\usepackage{amsthm}
\usepackage{comment}
\usepackage{xcolor}

\usepackage[colorlinks,pdfdisplaydoctitle]{hyperref}

\newtheorem{theorem}{Theorem}[section]

\newtheorem{corollary}[theorem]{Corollary}

\newtheorem{remark}[theorem]{Remark}

\newtheorem{algorithm}{Algorithm}

\definecolor{darkred}{rgb}{.7,0,0}

\definecolor{darkgreen}{rgb}{0,0.7,0}

\definecolor{darkblue}{rgb}{0,0,0.7}

\newcommand{\bu}{\bar{u}}
\newcommand{\KL}{Karhunen-Lo\'eve\,}
\newcommand{\Hp}{H^{\perp}}

\newcommand{\bbK}{\mathbb{K}}

\newcommand{\ud}{u^{\dagger}}
\newcommand{\etad}{\eta^{\dagger}}

\newcommand{\zj}{z^{(j)}}
\newcommand{\zzj}{{\widehat z}^{(j)}}
\newcommand{\ej}{\eta^{(j)}}
\newcommand{\uj}{u^{(j)}}
\newcommand{\pj}{p^{(j)}}
\newcommand{\uuj}{{\widehat u}^{(j)}}
\newcommand{\ppj}{{\widehat p}^{(j)}}

\newcommand{\ppl}{{\widehat p}^{(\ell)}}

\newcommand{\uuk}{{\widehat u}^{(k)}}
\newcommand{\ppk}{{\widehat p}^{(k)}}

\newcommand{\yj}{y^{(j)}}
\renewcommand{\dj}{d^{(j)}}
\newcommand{\pzj}{\tp^{(j)}}
\newcommand{\tp}{\widetilde p}

\newcommand{\uk}{u^{(k)}}

\newcommand{\pzk}{\tp^{(k)}}

\newcommand{\zb}{\overline{z}}
\newcommand{\pb}{\overline{p}}
\newcommand{\ub}{\overline{u}}

\newcommand{\cG}{\mathcal G}

\newcommand{\bbT}{\mathbb{T}}

\newcommand{\bbZ}{\mathbb{Z}}
\newcommand{\bbN}{\mathbb{N}}

\newcommand{\cA}{\mathcal{A}}
\newcommand{\cN}{\mathcal{N}}

\newcommand{\R} {\mathbb{R}}

\begin{document}
\title{Ensemble Kalman Methods for Inverse Problems}

\author{Marco A. Iglesias$\ast$, Kody J.H. Law$\ast$ and Andrew M. Stuart$\ast$}
\address{$\ast$ Mathematics Institute, University of Warwick, Coventry CV4 7AL, UK}
\ead{\{M.A.Iglesias-Hernandez,K.J.H.Law,A.M.Stuart\}@warwick.ac.uk}

\begin{abstract}

The Ensemble Kalman filter (EnKF)
was introduced by Evensen in 1994
\cite{evensen1994sequential} 
as a novel method for {\em data assimilation}: 
state estimation for noisily observed time-dependent problems.
Since that time it has had enormous impact in many application
domains because of its robustness and
ease of implementation, and numerical evidence of its
accuracy. In this paper we propose the application of an iterative ensemble Kalman method for the solution of a wide class of {\em inverse problems}. In this context we show that the estimate of the unknown function that we obtain with the ensemble Kalman method
lies in a subspace $\cA$ spanned by the initial ensemble.
Hence the resulting error may be bounded above by the error
found from the best approximation in this subspace. We provide
numerical experiments which compare the error incurred
by the ensemble Kalman method for inverse problems with the error of the
best approximation in $\cA$, and with variants on  
traditional least-squares approaches, 
restricted to the subspace $\cA$.
In so doing we demonstrate that the the ensemble Kalman method for inverse
problems provides a derivative-free optimization method
with comparable accuracy to that achieved by traditional 
least-squares approaches. Furthermore, we also
demonstrate that the accuracy is of the same order of magnitude
as that achieved by the best approximation. Three examples are used
to demonstrate these assertions: inversion of a compact
linear operator; inversion of piezometric head 
to determine hydraulic conductivity 
in a Darcy model of groundwater flow; and inversion of Eulerian velocity
measurements at positive times to determine the initial
condition in an incompressible fluid. 

\end{abstract}
\submitto{\IP}

\section{Introduction}\label{sec1}

Since its introduction in \cite{evensen1994sequential},
the Ensemble Kalman filter (EnKF) has had enormous impact
on applications of data assimilation to state and parameter estimation,
and in particular in oceanography \cite{evensen1996assimilation},
reservoir modelling \cite{EnKFReview} and
weather forecasting \cite{houtekamer2001sequential};
the books \cite{evensen2009data,kalnay2003atmospheric,Oliver}
give further details and references to applications
in these fields. Multiple variants of EnKF for state and parameter estimation in dynamic systems are available in the literature \cite{evensen2009data,kalnay2003atmospheric,EnKFReview}. In essence, all those techniques use an ensemble of states and parameters that is sequentially updated by means of the Kalman formula which blends the model and data available at a given time.

Motivated by ensemble Kalman-based approaches, in this paper we propose the application of an iterative ensemble Kalman method for the solution of inverse problems of finding $u$ given observations of the form
\begin{equation}
\label{eq:data}
y=\cG(u)+\eta,
\end{equation}
where $\cG: X \to Y$ is the {\em forward response
operator} mapping the unknown $u$ to the response/observation space. $X$ and $Y$ are Hilbert spaces, $\eta \in Y$ is a noise and $y \in Y$ the observed data. We assume that $\eta$ is an unknown realization of a mean zero random variable whose covariance $\Gamma$ is known to us. We are particularly interested in the case where $\cG$ is the forward response that arises from physical systems described by the solution of a PDE system. It is important to note, that in the abstract formulation of the inverse problem (\ref{eq:data}), both static and dynamic problems are considered in the same manner. For dynamic problems, the left hand side of (\ref{eq:data}) corresponds to all available observations which are, in turn, collected during a fixed time window contained in the time interval used for the underlying PDE formulation. Inverse problems of the type describe above are often ill-posed and their solution requires some sort of regularization \cite{engl1996regularization}. For the present work, regularization is introduced by incorporating prior knowledge of $u$ in the form of a finite dimensional (and hence
compact) set $\cA$ where the solution to (\ref{eq:data}) is sought (see \cite{Isakov} Chapter 2). The definition of the space  $\cA$ will be key in the formulation and properties of the ensemble method proposed for the solution of the inverse problem. 

In order to solve the inverse problem described above, artificial dynamics based on state augmentation are constructed. The state augmentation approach, typical for joint state and parameter estimation in the context of EnKF \cite{Anderson01anensemble}, can be applied in our abstract framework by constructing the space $Z=X \times Y$, and the mapping $\Xi: Z \to Z$ by
\begin{eqnarray*}
\Xi(z)=\left(
\begin{array}{c}
u\\
\cG(u)
\end{array}
\right),
\end{eqnarray*}
for $z \in Z$.  We define artificial dynamics by
\begin{eqnarray}
\label{eq:mod}
z_{n+1}&=\Xi(z_n). \label{eq:fwd}
\end{eqnarray}
Let us assume that data related to the artificial dynamics has the form
\begin{eqnarray}
y_{n+1}&=Hz_{n+1}+\eta_{n+1}. \label{eq:obs}
\end{eqnarray}
where the projection operator $H:Z\to Y$ is defined by  $H=[0,I]$ and $\{\eta_n\}_{n \in \bbZ^+}$ is an i.i.d. sequence with $\eta_1 \sim N(0,\Gamma)$ and $\Gamma$ defined above. In this paper we propose the application of the EnKF approach for state and parameter estimation for the artificial dynamical system (\ref{eq:mod}). EnKF uses an ensemble of particles that, at each iteration, is updated by combining the model (\ref{eq:mod}) with observational data via the standard Kalman update formula. In order to generate the data $\{y_n\}_{n \in \bbZ^+}$ required for the ensemble Kalman filter that we will apply to (\ref{eq:mod}), (\ref{eq:obs}) we perturb the single instance of the given observed data $y$ from (\ref{eq:data}) by independent realizations from the Gaussian random variable $N(0,\Gamma)$. We reemphasize that the 
iteration index $n$ in  (\ref{eq:mod}), (\ref{eq:obs}) is an artificial time;
in the case of a dynamic inverse problem, real time is 
contained in the abstract formulation of $\cG$ and
is not related to $n$.

 While the objective of standard EnKF approaches is to approximate, via an ensemble, statistical properties of a distribution conditioned to observations \cite{evensen2009data}, here the objective is to study a deterministic iterative scheme that aims at approximating the solution of the inverse problem (\ref{eq:data}) in the set $\cA$. We employ randomization of the single instance of the data $y$ given by (\ref{eq:data}) purely as a method to  move around the space $\cA$ in order to find improved approximations. More precisely, we construct an ensemble of interacting particles $\{z_{n}^{(j)}\}_{j=1}^{J}$ from which an estimate of the unknown is defined by
\begin{eqnarray}\label{eq:mean}
u_{n} \equiv \frac{1}{J}\sum_{j=1}^{J} u_{n}^{(j+1)}=\frac{1}{J}\sum_{j=1}^{J} \Hp z_{n}^{(j)}\,
\end{eqnarray}
where $\Hp: Z \to X$ is the projection operator defined by $\Hp=[I,0]$. We will show that, for all $n\in \mathbb{N}$, the ensemble $\{u_n^{j}\}_{j=1}^{J}$ remains in the set $\cA$ and, by
means of numerical examples we investigate the properties of $u_{n}$ at approximating, in the compact set $\cA$, the true unknown $\ud\in X$ which underlies the data. That is we assume that the data $y$ is given by
\begin{eqnarray}\label{eq:truth}
y=\cG(\ud)+\etad
\end{eqnarray}
for some noise $\etad \in Y.$ 

The purpose of the subsequent analysis is threefold:
(i) to demonstrate that the novel non-standard perspective
of the iterative ensemble Kalman method is a generic tool
for solving inverse problems; (ii) to provide 
some basic analysis of the properties of this algorithm for
inversion; (iii) to demonstrate numerically that the
method can be effective on a wide range of applications.

In \Sref{sec:PE} we introduce the iterative ensemble Kalman method for the solution to inverse problems. The space $\cA$ where a regularized solution of the inverse problem is sought is defined as the linear subspace generated by the initial ensemble members used for the iterative scheme. These, for the applications under consideration, can be generated from prior knowledge available in terms of a prior probability measure. The well-posedness of the algorithm is ensured by Theorem \ref{t:basic} where we prove that the method produces an approximation which lies in the subspace $\cA$. In other words, the approximation provided by the algorithm lies in the subspace spanned 
by the initial ensemble members, a fact observed
for a specific sequential implementation of the
EnKF in \cite{li2007iterative}. It is well known that the analysis step of the EnKF preserves the  subspace spanned by the ensemble \cite{EVENSENA}; we show that for the artificial dynamics (\ref{eq:mod}), (\ref{eq:obs}) the prediction step also preserves the subspace spanned by the ensemble leading to Theorem \ref{t:basic}. In Corollary \ref{c:basic} we use Theorem \ref{t:basic} to give a lower bound on the achievable approximation error of the ensemble Kalman algorithm. We then describe
two algorithms which we will use to evaluate
the ensemble Kalman methodology. The first is the least squares problem restricted to the subspace $\cA$. The second is the best approximation
of the truth in the subspace $\cA$. The best approximation is, of course course, not an implementable method as the truth
is not known; however it provides an important lower
bound on the achievable error of the ensemble Kalman method for synthetic experiments and
hence has a key conceptual role. At the end of \Sref{sec:PE} we discuss the links between the ensemble Kalman algorithm and the Tikhonov-Phillips regularized least squares solutions for the case of forward linear operators. 

Section \ref{sec:linear} contains numerical experiments
which illustrate the ideas in this paper on a linear
inverse problem. The forward operator is a compact operator
found from inverting the negative Laplacian plus identity with homogeneous boundary conditions. In
section \ref{sec6} we numerically 
study the groundwater flow inverse
problem of determining hydraulic conductivity
from piezometric head 
measurement in an elliptic Darcy flow
model. Section \ref{sec7} contains numerical results
concerning the problem of
determining the initial condition for the velocity in a
Navier-Stokes model of an incompressible fluid; the
observed data are pointwise (Eulerian) measurements of 
the velocity field. Conclusion and final remarks are presented in \Sref{Conclu}.

The numerical results in this paper all 
demonstrate that the iterative ensemble Kalman method for inversion is a
derivative-free regularized optimization technique
which produces numerical results similar in accuracy to
those found from least-squares based methods in 
the same subspace $\cA$. Furthermore, the three examples
serve to illustrate the point that the method offers
considerable flexibility through the choice of initial
ensemble, and hence the subspace $\cA$ in which it produces
an approximation. In particular, for the linear and
Darcy inverse problems, we make two choices of initial
ensemble: (i) draws from a prior Gaussian measure and
(ii) the \KL basis functions of the centered Gaussian
measure found by shifting
the prior by its mean. For the Navier-Stokes inverse problem the
initial ensemble is also chosen to comprise randomly 
drawn functions on the attractor of the dynamical system.

\section{An iterative ensemble Kalman method for inverse problems}
\label{sec:PE}

\subsection{Preliminaries}
\label{ssec:setup}

In the following, we use 
$\langle \cdot, \cdot \rangle$ and $\|\cdot\|$, etc. as the 
inner-product and norm on both $X$ and $Y$, and it will
be clear from the context which space is intended. Let $B^{-1}:D(B^{-1})\subset X\to X$ be a densely-defined unbounded self-adjoint operator with compact resolvent. Let us denote by $\{\lambda_{j}\}_{j=1}^{\infty}$ and $\{\phi_{j}\}_{j=1}^{\infty}$ the corresponding eigenvalues and eigenfunctions of $B^{-1}$. From standard theory it follows that
\begin{eqnarray}\label{eq:domain}
D(B^{-1})=\Big \{u\in X \vert ~\sum_{j=1}^{\infty} u_{j}^{2}\lambda_{j}^{2}<\infty\Big\}
\end{eqnarray}
and that $B^{-1}$ has the following spectral representation: $B^{-1}u=\sum_{j=1}^{\infty} \lambda_{j} u_{j}\phi_{j}$. We can additionally define the set
\begin{eqnarray}\label{eq:domain2}
D(B^{-1/2})=\Big \{u\in X \vert ~\sum_{j=1}^{\infty} u_{j}^{2}\lambda_{j}<\infty\Big\}
\end{eqnarray}
and consider the densely-defined operator $B^{-1/2}:D(B^{-1/2})\to X$ such that $B^{-1/2}u=\sum_{j=1}^{\infty} \lambda_{j}^{1/2} u_{j}\phi_{j}$. We also find it useful to define $\|\cdot\|_B \equiv \|B^{-1/2}\cdot\|$.

\subsection{The initial ensemble}
\label{ssec:alg}

The ensemble Kalman method uses an ensemble of particles $\{z_{n-1}^{(j)}\}_{j=1}^{J}$ which, at the $n$ iteration level, is updated by combining the artificial dynamics (\ref{eq:mod}) with artificial data $y_{n}$ obtained from perturbing our original data (\ref{eq:data}) in order to obtain a new ensemble $\{z_{n}^{(j)}\}_{j=1}^{J}$ from which the estimate of the unknown (\ref{eq:mean}) is computed. The scheme requires an initial (or first guess) ensemble of particles $\{z_{0}^{(j)}\}_{j=1}^{J}$ which will be iteratively updated with the EnKF method described below. The ensemble $\{z_{0}^{(j)}\}_{j=1}^{J}$ can be defined by constructing an ensemble $\{ \psi^{(j)}\}_{j=1}^J$ in $X$. Once $\{ \psi^{(j)}\}_{j=1}^J$ is specified, we can simply define
\begin{eqnarray*}
z_{0}^{(j)}=\left(
\begin{array}{c}
\psi^{(j)}\\
\cG(\psi^{(j)})
\end{array}
\right),
\end{eqnarray*}
and so our first guess of the iterative scheme $u_{0}$ (see expression (\ref{eq:mean})) is simply the mean of the initial ensemble in the space of the unknown. The construction of the initial ensemble $\{ \psi^{(j)}\}_{j=1}^J$ is in turn related to the definition of the space $\cA$ where the solution to the inverse problem is sought. Clearly, $ \psi^{(j)}$ must belong to the compact set $\cA$  which regularizes the inverse problem by incorporating prior knowledge. For the applications described in \Sref{sec:linear}, \Sref{sec6} and \Sref{sec7}, we assume that prior knowledge is available in terms of a prior probability measure that we denote by $\mu_0$. Given this prior distribution, we construct the initial ensemble $\{ \psi^{(j)}\}_{j=1}^J$ defined as $\psi^{(j)} \sim \mu_0$ i.i.d. for some $J<\infty$. Then, for consistency we define
\begin{eqnarray}\label{eq:A}
\cA={\rm span} \{ \psi^{(j)}\}_{j=1}^J
\end{eqnarray} 
comprised of the initial ensemble. Note that if $\mu_{0}$ is Gaussian $N(\bar{u},C)$, we may additionally consider $\cA$ with defined $\psi^{(j)} = \bar{u} + \sqrt{\lambda_j}\phi_j$ where $(\lambda_j,\phi_j)$ denote
eigenvalue/eigenvector pairs of $C$, in descending order by
eigenvalue -- this is the \KL basis. 

For the experiments of \Sref{sec:linear} and \Sref{sec6}, Gaussian priors are considered. For the Navier-Stokes example of \Sref{sec7}, the prior will be the empirical measure supported on the attractor. For the latter case an empirical covariance will be used to construct the KL basis. In summary, the proposed approach for solving inverse problems will be tested with an initial prior ensemble generated from  (a) $\psi^{(j)} \sim \mu_0$  i.i.d. and all algorithms using this choice of $\cA$ will bear the subscript $R$ for random; (b) $\psi^{(j)} = \bar{u} + \sqrt{\lambda_j}\phi_j$, $j \leq J$, and all algorithms using this method will bear the subscript ${\rm KL}$ for \KL. 

We emphasize that the initial ensemble $\{ \psi^{(j)}\}_{j=1}^J$ and therefore the definition of $\cA$ is a design parameter aiming at incorporating prior knowledge relevant to the application. For example, in the case not considered here, where no underlying prior probability distribution is prescribed, the initial ensemble can be defined in terms of a truncated basis $\{ \psi^{(j)}\}_{j=1}^J$ of $X$. Regardless of how the initial ensemble is chosen, the proposed approach. The proposed approach will then find a solution to the inverse problem in the subspace $\cA$ defined by (\ref{eq:A}).

\subsection{Iterative ensemble Kalman method for inverse problems}
\label{ssec:alg}

The iterative algorithm that we propose for solution of the inverse problem (\ref{eq:data}) is the following
\begin{algorithm}{Iterative ensemble method for inverse problems.}\\
 Let $\{z_{0}^{(j)}\}_{j=1}^{J}$ be the initial ensemble.\\
For $n=1,\dots$
\begin{itemize}
\item[(1)] \textbf{Prediction step.} Propagate, under the artificial dynamics (\ref{eq:fwd}), the ensemble of particles 
\begin{equation}
\label{eq:fwdj}
\zzj_{n+1}=\Xi(\zj_n).
\end{equation}
From this ensemble we define a sample mean
and covariance as follows:
\begin{eqnarray}
\zb_{n+1}&=\frac{1}{J}\sum_{j=1}^J \zzj_{n+1}\label{eq:mean}\\
C_{n+1}&=\frac{1}{J}\sum_{j=1}^J \zzj_{n+1}(\zzj_{n+1})^T-\zb_{n+1}\zb_{n+1}^T.\label{eq:cov}
\end{eqnarray}
\item[(2)]\textbf{Analysis step.}  Define the Kalman gain $K_n$ by 
\begin{eqnarray}
K_n&=C_nH^T (HC_nH^T+\Gamma)^{-1}, \label{eq:g2}
\end{eqnarray}
where $H^T$ is the adjoint operator of $H\equiv [0,I]$. Update each ensemble member as follows
\begin{eqnarray}
\label{eq:update}
\zj_{n+1}&=&I\zzj_{n+1}+K_{n+1}(\yj_{n+1}-H\zzj_{n+1})\\
&=&(I-K_{n+1}H)\zzj_{n+1}+K_{n+1}\yj_{n+1}.
\end{eqnarray}
where
\begin{equation}
\label{eq:Y1}
\yj_{n+1}=y+\ej_{n+1}.
\end{equation}
and the $\ej_{n+1}$ are an i.i.d collections of
vectors indexed by $(j,n)$ with $\eta_1^{(1)} \sim N(0,\Gamma).$
\end{itemize}
\item[(3)] Compute the mean of the parameter update
\begin{eqnarray}\label{eq:update}
u_{n+1} \equiv \frac{1}{J}\sum_{j=1}^{J} u_{n+1}^{(j+1)}
\end{eqnarray}
and check for convergence (see discussion below).
\end{algorithm}

Each iteration of the ensemble Kalman algorithm breaks into two parts, a {\em prediction step}
and an {\em analysis step}. The prediction step maps the current ensemble of
particles into the data space, and thus introduces
information about the forward model. The analysis step
makes comparisons of the mapped ensemble, in the data
space, with the data, or with versions of the data perturbed with noise; it is at this stage that the
ensemble is modified in an attempt to better match the
data. As noted earlier, we generate artificial data (\ref{eq:obs}) consistent with the artificial dynamics (\ref{eq:mod}) by perturbing the observed data (\ref{eq:data}). More precisely, data is perturbed according to  (\ref{eq:Y1}) with noise consistent with the distribution assumed
on the noise $\eta$ in the inverse problem (\ref{eq:data}). Perturbing the noise in the standard EnKF methods is typically used to properly capture statistical properties. However, in the present application, we are merely interested in a deterministic estimation of the inverse problem. Nonetheless, numerical results  (not shown) without perturbing the noise (e.g. with $\ej_{n+1}=0$ in (\ref{eq:Y1})) gave rise to less accurate solutions than the ones obtained when the data was perturbed according to (\ref{eq:Y1}). The added noise presumably helps
the algorithm explore the approximation space and hence
find a better approximation within it.

The proper termination of iterative regularization techniques
\cite{Iterative} is essential for the regularization of ill-posed
inverse problems. The discrepancy principle, for example, provides a
stopping criterion that will ensure the convergence and regularizing
properties of iterative techniques such as the Landweber iteration,
Levenberg-Marquardt and the iteratively regularized Gauss-Newton
\cite{Iterative}. A complete analysis of the convergence and
regularizing properties of the iterative ensemble Kalman method is
beyond the scope of this paper. Nonetheless, numerical experiments
suggest that the discrepancy principle can be a useful stopping
criterion for the ensemble Kalman algorithm; in particular these
experiments indicate that the criterion ensures the stable computation
of an approximation to the inverse problem
(\ref{eq:data}). Concretely, according to the discrepancy principle,
the ensemble Kalman method is terminated for the first $n$ such that
$\vert\vert y-\cG(u_{n})\vert\vert_{\Gamma}\leq \tau \vert\vert
\eta^{\dagger}\vert\vert_{\Gamma}$ for some $\tau>1$ and where
$\eta^{\dagger}$ is the realization of noise associated to the data (\ref{eq:truth}). 

We show in the next section that the entire ensemble at each
step of the algorithm lies in the set $\cA$ defined
by (\ref{eq:A}) and that, hence, the parameter estimate (\ref{eq:update}) lies
in the set $\cA$.

\subsection{Properties of the iterative ensemble method }

We wish to show that each estimate $u_{n}$ (\ref{eq:update}) of the ensemble Kalman method is a linear of combinations of the initial ensemble
with nonlinear weights reflecting the observed data, which from (\ref{eq:A}) implies $u_{n}\in\cA$. First, note that all the vectors and operators involved have block
structure inherited from the structure of the space $Z=X \times Y.$
For example we have
\begin{eqnarray*}
\zzj_{n+1}
=\left(
\begin{array}{c}
\uuj_{n+1}\\
\ppj_{n+1}
\end{array}
\right)
=\left(
\begin{array}{c}
\uj_{n}\\
\cG\bigl(\uj_n\bigr)
\end{array}
\right), \quad
\zj_{n}=\left(
\begin{array}{c}
\uj_{n}\\
\pj_{n}
\end{array}
\right).
\end{eqnarray*}
We also have
\begin{eqnarray*}
\zb_n=\left(
\begin{array}{c}
\ub_n\\
\pb_n
\end{array}
\right), \quad
C_n=\left(
\begin{array}{cc}
C^{uu}_n & C^{up}_n\\
(C^{up}_n)^T & C^{pp}_n
\end{array}
\right).
\end{eqnarray*}
The vectors $\ub_n$ and $\pb_n$ are given by
\begin{eqnarray}
\ub_n&=\frac{1}{J}\sum_{j=1}^J \uuj_n=\frac{1}{J}\sum_{j=1}^J \uj_{n-1}\label{eq:meanu}\\
\pb_n&=\frac{1}{J}\sum_{j=1}^J \ppj_n=\frac{1}{J}\sum_{j=1}^J \cG\bigl(\uj_{n-1}\bigr)\label{eq:meanp}
\end{eqnarray}

The blocks within $C_n$ are given by
\begin{eqnarray}
C^{uu}_n&=\frac{1}{J}\sum_{j=1}^J \uuj_n(\uuj_n)^T-\ub_n\ub_n^T,\label{eq:cuu}\\
C^{up}_n&=\frac{1}{J}\sum_{j=1}^J \uuj(\ppj)^T-\ub_n\pb_n^T,\label{eq:cup}\\
C^{pp}_n&=\frac{1}{J}\sum_{j=1}^J \ppj(\ppj)^T-\pb_n\pb_n^T,\label{eq:cpp}.
\end{eqnarray}

As we indicated before, it is well-known that the analysis step of the
Ensemble Kalman filter provides an updated ensemble
which is in the linear span of the forecast ensemble \cite{EVENSENA}.
In the context of our iterative method for inverse
problems the forecast ensemble itself is in the linear
span of the preceding analysis ensemble, when projected
into the parameter coordinate. Combining these two
observations shows that the ensemble parameter estimate
lies in the linear span of the initial ensemble.

\begin{theorem}
\label{t:basic}
For every $(n,j) \in \bbN \times \{1,\cdots,J\}$
we have $u_{n+1}^{(j)} \in \cA$ and hence $u_{n+1} \in \cA$ for all $n\in \bbN$. 
\end{theorem}

\begin{proof} This is a straightforward induction, using
the properties of the update formulae. Clearly the statement
is true for $n=0$. Assume that it is true for $n$. 
The operator $K_n$ have particular structure 
inherited from the form of $H$. In concrete, we have
\begin{eqnarray*}
K_n=\left(
\begin{array}{c}
C^{up}_n(C^{pp}_n+\Gamma)^{-1}\\
C^{pp}_n(C^{pp}_n+\Gamma)^{-1}
\end{array}
\right).
\end{eqnarray*}
Note that

\begin{equation}\label{eq:S2.1}
C^{up}_n=\frac{1}{J}\sum_{j=1}^J \uuj_n(\pzj_n)^T, \qquad 
C^{pp}_n=\frac{1}{J}\sum_{j=1}^J \ppj_n(\pzj_n)^T
\end{equation}
where
\begin{equation}\label{eq:S2.2}
\pzj_n=\ppj_n-\frac{1}{J}\sum_{\ell=1}^J \ppl_n.
\end{equation}

Recall that, from the structure of the map $\Xi$, we
have
\begin{equation}\label{eq:S2.3}
\uuj_{n+1}=\uj_n, \quad \ppj_{n+1}=\cG(\uj_n)
\end{equation}

and from the definition of $H$ we get that
\begin{eqnarray*}
K_{n}H=\left(
\begin{array}{cc}
0 & C^{up}_n (C^{pp}_n+\Gamma)^{-1}\\
0 & C^{pp}_n (C^{pp}_n+\Gamma)^{-1}
\end{array}
\right)
\end{eqnarray*}
Using these facts in (\ref{eq:update}) we deduce that
the update equations are
\begin{eqnarray}
\uj_{n+1}&=\uj_{n}+C^{up}_{n+1}(C^{pp}_{n+1}+\Gamma)^{-1}\Bigl(\yj_{n+1}-
\cG\bigl(\uj_n\bigr)\Bigr)\label{eq:uu} \\
\pj_{n+1}&=\cG(\uj_{n})+C^{pp}_{n+1}(C^{pp}_{n+1}+\Gamma)^{-1}
\Bigl(\yj_{n+1}-\cG\bigl(\uj_n\big)\Bigr)\label{eq:pp}.
\end{eqnarray}

If we define
$$\dj_{n+1}=(C^{pp}_{n+1}+\Gamma)^{-1}\Bigl(\yj_{n+1}-
\cG\bigl(\uj_n\bigr)\Bigr)$$ 
then the update formula (\ref{eq:uu}) 
for the unknown $u$ may be written as
\begin{eqnarray}
\uj_{n+1} 
&=\uj_{n}+\frac{1}{J}\sum_{k=1}^J \langle 
\pzk_{n+1}, \dj_{n+1} \rangle \uuk_{n+1}\\
&=\uj_{n}+\frac{1}{J}\sum_{k=1}^J \langle 
\pzk_{n+1}, \dj_{n+1} \rangle \uk_n.
\label{eq:useful1}
\end{eqnarray}
In view of the inductive hypothesis at step $n$, this demonstrates
that $\uj_{n+1} \in \cA$ for all $j\in \{1,\cdots, J\}$. Then, from (\ref{eq:update}) and (\ref{eq:A}) it follows that $u_{n+1}\in \cA$ which finalizes the proof.
\end{proof}

\begin{remark}
\label{rem:basic}
The proof demonstrates that the solution at step $n$ is simply
a linear combination of the original samples; however
the coefficients in the linear combination depend nonlinearly
on the process.
Note also that the update formula (\ref{eq:pp}) for the system
response may be written as
\begin{eqnarray}
\pj_{n+1} 
&=\cG\bigl(\uj_n\bigr)+\frac{1}{J}\sum_{k=1}^J \langle 
\pzk_{n+1}, \dj_{n+1} \rangle \ppk_{n+1}\\
&=
\cG\bigl(\uj_n\bigr)
+\frac{1}{J}\sum_{k=1}^J \langle 
\pzk_{n+1}, \dj_{n+1} \rangle \cG\bigl(\uk_n).
\label{eq:useful2}
\end{eqnarray}
\end{remark}

We have the following lower bound for the accuracy
of the estimates produced by the ensemble Kalman algorithm:

\begin{corollary}
\label{c:basic}
The error between the estimate of $u$ at time $n$ and the truth $\ud$
satisfies
$$\|u_n-\ud\| \ge \inf_{v \in \cA}\|v-\ud\|.$$
\end{corollary}

\subsection{Evaluating the performance of the iterative Kalman method for inverse problems}

With the iterative EnKF-based algorithm previously described we aim at finding solutions $u_{EnKF}$ of the inverse problem (\ref{eq:data}) in the subspace $\cA$. We now wish to evaluate the performance of $u_{EnKF}$ at recovering the truth $u^{\dagger}$ defined in (\ref{eq:truth}). Central to this task is the misfit functional
\begin{equation}
\label{eq:misfit}
\Phi(u)= \|y-\cG(u)\|^2_\Gamma,
\end{equation}
which should be minimized in some sense. As we indicated earlier, the inversion of $\cG$ is ill-posed and so minimization of $\Phi$ 
over the whole space $X$ is not possible. As for the Ensemble Kalman method, we choose to regularize the least squares methods by incorporating prior knowledge via minimization of $\Phi$ over the compact set $\cA\in X $. Then, we compute the {\it least squares} solution
\begin{equation}\label{eq:ls}
u_{LS}= {\rm argmin}_{u \in \cA} ||y-\cG(u)||^2_{\Gamma}
\label{eq:pedls}
\end{equation}
or generalizations to include truncated Newton-CG iterative methods \cite{hanke1997regularizing} as well as Tikhonov-Phillips regularization
\cite{vogel2002computational},
based on $\mu_0=N(\bu,C)$. In other words, we consider $u_{TP}= {\rm argmin}_{u \in \cA} ||y-\cG(u)||^2_{\Gamma}+\|u-\bu\|_C^2.$

We recall our definition to the true solution $u^{\dagger}$ (see equation \ref{eq:truth})) to the inverse problem (\ref{eq:data}). Given that both the ensemble Kalman method and the least-squares solution (\ref{eq:ls}) provide, in the subspace $\cA$, some approximation to the truth, it is then natural to consider the {\em best approximation}
to the truth in $\cA$. In other words, 
\begin{equation}
u_{BA} = {\rm argmin}_{u \in \cA} ||u-u^\dagger||^2_{\Gamma}.
\label{eq:pedba}
\end{equation}
Best approximation properties of ensemble methods is discussed in a
different context in 
\cite{Perianez}.

In the experiments below we compare the accuracy in the approximations obtained with the ensemble Kalman method with respect to the least square solution and the best approximation. The latter, however, is done only for the sake of assessing the performance the technique with synthetic experiments for which the truth is known.

In typical applications of EnKF-based methods the amount of observational data is usually much smaller than the dimensions of the space $X$ for the unknown. Therefore, the matrix computations for the construction of the Kalman gain matrix  (\ref{eq:g2}) are often negligible compared to the computational cost of evaluating, for each ensemble member, the forward model (\ref{eq:mod}) at each iteration of the scheme. In this case, the computational cost of the ensemble Kalman algorithm is dominated by the size of the ensemble multiplied by the total number of iterations. While the computational cost of EnKF is quite standard, the cost of solving the least-squares problem depends on the particular implementation used for the application under consideration. Since the optimality of the implementation of least-squares problems is beyond the scope of our work, we do not assess the computational efficiency of the ensemble method with respect to the optimization methods used for the solution of the least-squares problem.

\subsection{Connection between regularized least-squares and the iterative ensemble Kalman method for linear inverse problems}

\label{ssec:lp}

It is instructive to consider the case where $\cG(u)=Gu$
for some linear operator $G:X \to Y$ as this will enable us to
make links between the iterative ensemble Kalman method and the standard regularized least squares problems. Let $C^{-1}$ be an operator like the one defined in \Sref{ssec:setup} with $D(C^{-1/2})$ defined analogously to (\ref{eq:domain2}). Consider the Tikhonov-Phillips regularized functional 
\begin{equation}
\label{eq:misfit3}
I(u)= \|y-Gu\|^2_\Gamma+\|u-\bu\|_C^2.
\end{equation}
If $\|y-G\cdot\|^2_\Gamma$ is continuous on $D(C^{-1/2})$ then $I(u)$ is weakly lower semicontinuous in $X$ and hence the unique minimizer $u_{\rm TP}={\rm argmin}_{u \in D(C^{-1/2})} I(u)$ is attained in $D(C^{-1/2})$; indeed minimizing sequences converge strongly in $D(C^{-1/2})$ along a subsequence. The existence of a minimizer
follows from the standard theory of calculus of variations,
see Theorem 1, Part 2, in \cite{Dac}. Uniqueness follows
from the quadratic nature of $I$. The fact that minimizing
sequences  converge strongly uses an argument from 
Theorem II.2.1 in \cite{Sta}, as detailed in Theorem 2.7
of \cite{cotter2009bayesian}. We note that
minimization of $I$ given by (\ref{eq:misfit3})
has a statistical interpretation as the maximum a posteriori (MAP) estimator
for Bayesian solution to the inverse problem with
Gaussian prior $N(\bu,C)$ on $u$ \cite{kaipio2005statistical}. Furthermore the minimizer of (\ref{eq:misfit3}) can be computed explicitly under a variety
of different assumptions on the linear operators $C$, $G$ and
$\Gamma$. In all cases the formal expression for the
solution can then be written \cite{lehtinen1999linear,mandelbaum1984linear}
\begin{equation}\label{eq:TP}
u_{\rm TP} = \bu + C G^* ( G C G^* + \Gamma)^{-1} (y - G \bu).
\label{kalman}
\end{equation}
The purpose of this subsection is now to show the connection between the iterative Kalman method that we introduce in this paper and the regularized least-squares solution (\ref{eq:TP}) for the solution of linear inverse problems that arise from linear forward operators. Recall the set ${\cal A}$ defined by (\ref{eq:A}). We consider the case where $\psi^{(j)}\sim \mu_0$. Let $\mu_{0}= N(\bu,C)$ and consider $n=0$ in the algorithm of the previous section. Let us define the prior ensemble mean and covariance 
\begin{eqnarray}
m_{J}\equiv \frac{1}{J}\sum_{j=1}^J u_{0}^{(j)}, \qquad C_{J}\equiv \frac{1}{J-1}\sum_{j=1}^J (\uj_1-m_{J})(\uj_1- m_{J})^T
\end{eqnarray}
From (\ref{eq:S2.1})-(\ref{eq:S2.3}) for $n=0$ and $\cG(u)=Gu$ we have
\begin{eqnarray}
\fl C^{up}_1=\frac{1}{J}\sum_{j=1}^J \uuj_0(\pzj_0)^T=\frac{1}{J}\sum_{j=1}^J  \uj_0(G \uj_0-\frac{1}{J}\sum_{\ell=1}^J G u_0^{(l)})^T=\frac{1}{J}\sum_{j=1}^J \uj_0(G\uj_0-G m_{J})^T\nonumber\\
\fl =\frac{1}{J}\sum_{j=1}^J (\uj_0-m_{J})(G\uj_0-G m_{J})^T=\frac{1}{J}\sum_{j=1}^J (\uj_0-m_{J})(\uj_0- m_{J})^TG^{*}\nonumber\\=\Bigg[\frac{J-1}{J}\Bigg]C_{J} G^{*}
\end{eqnarray}
Similarly,
\begin{eqnarray}
\fl C^{pp}_1=\frac{1}{J}\sum_{j=1}^J \ppj_0(\pzj_0)^T=\frac{1}{J}\sum_{j=1}^J G \uj_0 (G \uj_0-\frac{1}{J}\sum_{\ell=1}^J G u_{0}^{(l)})^T=\frac{1}{J}\sum_{j=1}^J G \uj_0 (G \uj_0-G m_{J})^T\nonumber\\
\fl =\frac{1}{J}\sum_{j=1}^JG (  \uj_0 -m_{J} )( \uj_0-m_{J})^TG^{*}
=\Bigg[\frac{J-1}{J}\Bigg]G C_{J}G^{*}
\end{eqnarray}
Therefore, from (\ref{eq:uu}), (\ref{eq:update}) and (\ref{eq:Y1}) it follows that the $u_{1}$ estimate of the iterative method is 
\begin{eqnarray}
\fl u_{1}\equiv \frac{1}{J}\sum_{j=1}^{J}\uj_{1}=  \frac{1}{J}\sum_{j=1}^{J}\uj_{0}+C^{up}_{1}(C^{pp}_{1}+\Gamma)^{-1}\Bigl( \frac{1}{J}\sum_{j=0}^{J}(\yj_{1}-
G \uj_0)\Bigr)\nonumber\\
= m_{J}+C^{up}_{1}(C^{pp}_{1}+\Gamma)^{-1}\Bigl(y+\frac{1}{J}\sum_{j=0}^{J}\eta_{1}^{(j)}-G m_{J}
\Bigr)\label{eq:uu2} 
\end{eqnarray}
Thus, the estimate $u_{1}$ provides an approximation to the regularized least-squares problem (\ref{eq:TP}) that converges in the limit of $J\to \infty$ (i.e. with infinite number of ensemble members). Indeed, notice that, almost surely as $J \to \infty$
\begin{eqnarray}
\fl \overline{u}_{J} \to  \overline{u}, \qquad C_{J}\to C,\qquad C^{up}_1 \to  CG^{*},\qquad  C^{pp}_1\to G C G^{*},\qquad \frac{1}{J}\sum_{j=0}^{J}\eta_{1}^{(j)}\to 0.
\end{eqnarray}
Therefore, as $J\to \infty$
\begin{eqnarray}
u_{1} \to u_{\rm TP} = \bu + C G^* ( G C G^* + \Gamma)^{-1} (y - G \bu).
\end{eqnarray}
This link between the ensemble Kalman method and the regularized least square-problems for the linear inverse problem opens the possibility of solving nonlinear inverse problems by iterating the Kalman filter as we proposed in the subsection \ref{ssec:alg}. However, it is important to remark that, in contrast to the least-squares approach, the implementation of the ensemble Kalman method does not require the derivative of the forward operator. 

\section{Elliptic Equation}
\label{sec:linear}

As a simple pedagogical example, we consider the ill-posed inverse
problem of recovering the right-hand side of an elliptic equation in 
one spatial dimension
given noisy observation of the solution.  This explicitly
solvable linear model will allow us to elucidate 
the performance of the ensemble Kalman method as an iterative regularization method.  
The results show that ensemble Kalman method performs comparably to 
Tikhonov-Phillips regularized least squares for this problem.
The lower bound BA is always significantly
better, but of course unimplementable in practice,
since the truth is unknown.

\subsection{Setting}

Consider the one dimensional elliptic equation
\begin{eqnarray*}
-\frac{d^2 p}{dx^2}+p &=& u\\
u(0)=u(\pi)&=&0.
\end{eqnarray*}
Thus $G=A^{-1}$ where $A=(-\frac{d^2}{dx^2}+1)$ and
$D(A)=H^2(I)\cap H^1_0(I)$ with $I=(0,\pi).$
We are interested in the inverse problem of recovering 
$u$ from noisy observations of $p$:
\begin{eqnarray*}
y&=&p+\eta\\
&=&A^{-1}u+\eta.
\end{eqnarray*}
For simplicity we assume that the noise is white:
$\eta \sim N(0,\gamma^2 I).$
We consider Tikhonov-Phillips regularization of the form
$\|u\|^2_C$, where $C=\beta(A-I)^{-1}$.  
The problem may be solved explicitly in the Fourier
sine basis, and the coefficients of the Tikhonov-Phillips
regularized least squares solution, in this basis, $u_k$ may
be expressed in terms of the coefficients of the data  in the
same basis, $y_k$: 
\begin{equation}
\left [ \left ( \frac{1}{\gamma (1+k^2)}  \right )^2 + \beta^{-1}k^2 \right ] u_k =
\frac{y_k}{\gamma^2 (1+k^2)}, \quad
k=1,...,\infty
\label{pedreg}
\end{equation}
This demonstrates explicitly the regularization
which is present for wavenumbers $k$ such that
$k^6 \ge {\cal O}(\beta\gamma^{-2}).$
We now present some detailed numerical experiments
which will place the ensemble Kalman algorithm in the context of an
iterative regularization scheme.  This will provide
a roadmap for understanding the nonlinear inverse problem
applications we explore in subsequent sections.

\subsection{Numerical Results}

Throughout this subsection we choose $\beta=10$ and
$\gamma=0.01.$ We choose a truth $u^\dagger \sim N(0,C)$
and simulate data from the model: 
$y = A^{-1}u^\dagger + \eta^\dagger$, 
where $\eta^\dagger \sim \cN(0,\Gamma)$. 
Recall from Section \ref{sec:PE} that the space 
$\cA={\rm span} \{ \psi^{(j)}\}_{j=1}^J$
will be chosen either based on draws from $N(0,C)$
(with subscript $R$) or from the \KL basis for $N(0,C)$
(with subscript ${\rm KL}$). The experiments obtained from the
iterative ensemble Kalman method with the set $\cA$ chosen according
to the two scenarios described above are denoted by EnKF$_{\rm R}$ and
EnKF$_{\rm KL}$, respectively.  In Figure \ref{ped_ensembles} (left)
we display the relative error with respect to the truth of EnKF$_{\rm
  R}$ and EnKF$_{\rm KL}$ (i.e. $\vert\vert
u_{EnKF}-u^{\dagger}\vert\vert/\vert\vert
u^{\dagger}\vert\vert$). 
The data misfit $\vert\vert
y-G(u_{EnKF})\vert\vert$ is shown in Figure
\ref{ped_ensembles} (right). The black dotted line corresponds to the
value of the noise level, i.e. 
$\vert\vert y-G(u^{\dagger})\vert\vert=\vert\vert
\eta^\dagger\vert\vert$. 
Note that the error of EnKF$_{\rm KL}$ decreases for some iterations before reaching its minimum, while EnKF$_{\rm R}$ reaches its minimum after a small number of iterations (often only one), and then increases. In both cases the error reaches its minimum at an iteration step such that the value of the associated data misfit is approximately the noise level defined above. In particular, for EnKF$_{\rm R}$ we observe that once $\vert\vert y^{\dagger}-G(u_{EnKF})\vert\vert$ is below the noise, the error in the estimate increases. This behavior has been often reported when some optimization techniques are applied for the solution of inverse ill-posed problems  \cite{Nagy}. It is clear that, once the data misfit is at, or
below, the noise level, the given choice of $\cA$ does not provide sufficient regularization of the problem. Indeed this suggests that an early termination based 
on the discrepancy principle \cite{hanke1997regularizing} may furnish the ensemble Kalman algorithm with the regularizing properties needed for this choice of $\cA$. It is worth mentioning that the aforementioned increase in the error after the data misfit reaches the noise level was observed in additional experiments (not shown) where the data was generated with different noise levels. 

It is clear form Figure \ref{ped_ensembles} that the selection of $\cA$ with elements from the first elements of the KL basis alleviates the ill-posedness of the inverse problem. More precisely, those first elements of the KL basis corresponds to the largest eigenvalues of the covariance operator $C$. Therefore, the subspace $\cA$ where the solution of the ill-posed problem is sought does not contain the directions associated with smaller eigenvalues which are, in turn, responsible for the lack of stability of the inversion. Then, in contrast to the EnKF$_{R}$, the estimate generated EnKF$_{KL}$ is a linear combination of eigenfunctions associated with larger eigenvalues of $C$.

For the two scenarios described earlier, we now compare the performance of the ensemble Kalman method with respect to the Tikhonov-Phillips regularized least-squares (LS) and the best approximation BA methods.  Let us consider first the ``$R$'' random scenario where $\cA$ is the linear subspace generated from random draws from $\mu_{0}$. Since the estimate depends on the choice of $\cA$, we compare EnKF$_{\rm R}$, LS$_{\rm R}$ and BA$_{\rm R}$ on 100 different $\cA$'s corresponding to 100 different prior ensembles. Furthermore, since we noted that EnKF$_{\rm R}$ increases after the first iterations, in this example we consider only the first iteration of EnKF$_{\rm R}$. The relative errors of the three estimator for different $\cA's$ are displayed in Figure \ref{ped_ensembles2}. The methods clearly indicates that the ensemble Kalman method is comparable to the least-squares method in terms of accuracy,
but does not involve derivatives of the forward operator. Both the ensemble Kalman method and the least-squares are
less accurate than the best approximation, of course, but produce errors
of similar order of magnitude. In the first column of Table \ref{Table} we display the values of the aforementioned estimators averaged over the sets $\cA$ (generated from the prior). For the case where $\cA$ is generated from the KL basis, the comparison of the EnKF$_{KL}$, LS$_{KL}$ and $BA_{KL}$ is straightforward and the values are also displayed in the first column of Table \ref{Table}. Note that for EnKF$_{KL}$ we consider the estimate obtained from at the last iteration. These results again show that EnKF and LS are similar in terms
of accuracy, and produce errors of similar order of magnitude to
BA.

\begin{table}
\caption{Relative errors with respect to the truth for the experiments in \Sref{sec:linear} (elliptic), \Sref{sec6} (groundwater) and \Sref{sec7} (NSE)}
\begin{center} \footnotesize
\item[] \lineup
\begin{tabular}{|c|c|c|c|} \hline
Method &Elliptic& Groundwater &NSE\\ \hline
\lower.3ex\hbox{$EnKF_{R}$ (averaged over $\cA$)} & \lower.3ex\hbox{0.257} &\lower.3ex\hbox{0.597}& \lower.3ex\hbox{0.661}\\
\lower.3ex\hbox{$LS_{R}$  (averaged over $\cA$)} &\lower.3ex\hbox{0.264} &\lower.3ex\hbox{0.581} &\lower.3ex\hbox{0.591}\\
\lower.3ex\hbox{$BA_{R}$  (averaged over $\cA$)} &\lower.3ex\hbox{0.111} &\lower.3ex\hbox{0.367} &\lower.3ex\hbox{0.499}\\
\lower.3ex\hbox{$EnKF_{KL}$ (final iteration $n=30$)} & \lower.3ex\hbox{0.270} &\lower.3ex\hbox{0.591}& \lower.3ex\hbox{0.650}\\
\lower.3ex\hbox{$LS_{KL}$ } &\lower.3ex\hbox{0.250}&\lower.3ex\hbox{0.569} &\lower.3ex\hbox{0.500}\\
\lower.3ex\hbox{$BA_{KL}$ } &\lower.3ex\hbox{0.070}&\lower.3ex\hbox{0.278} &\lower.3ex\hbox{0.439}\\
\hline
\end{tabular}
\end{center}\label{Table}
\end{table}

\begin{figure}[h]
\begin{center}
     \includegraphics[width=.47\textwidth]{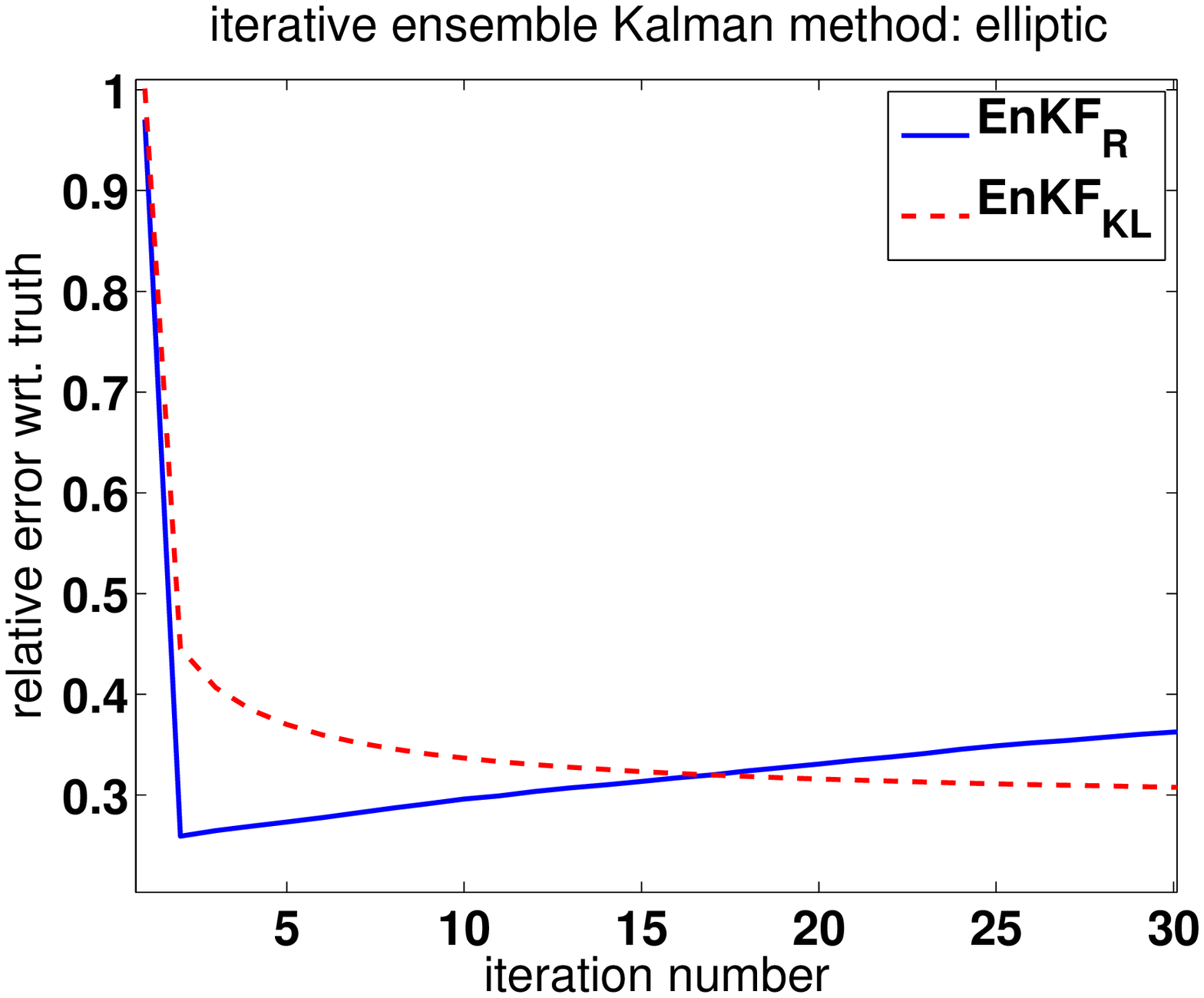}
     \includegraphics[width=.47\textwidth]{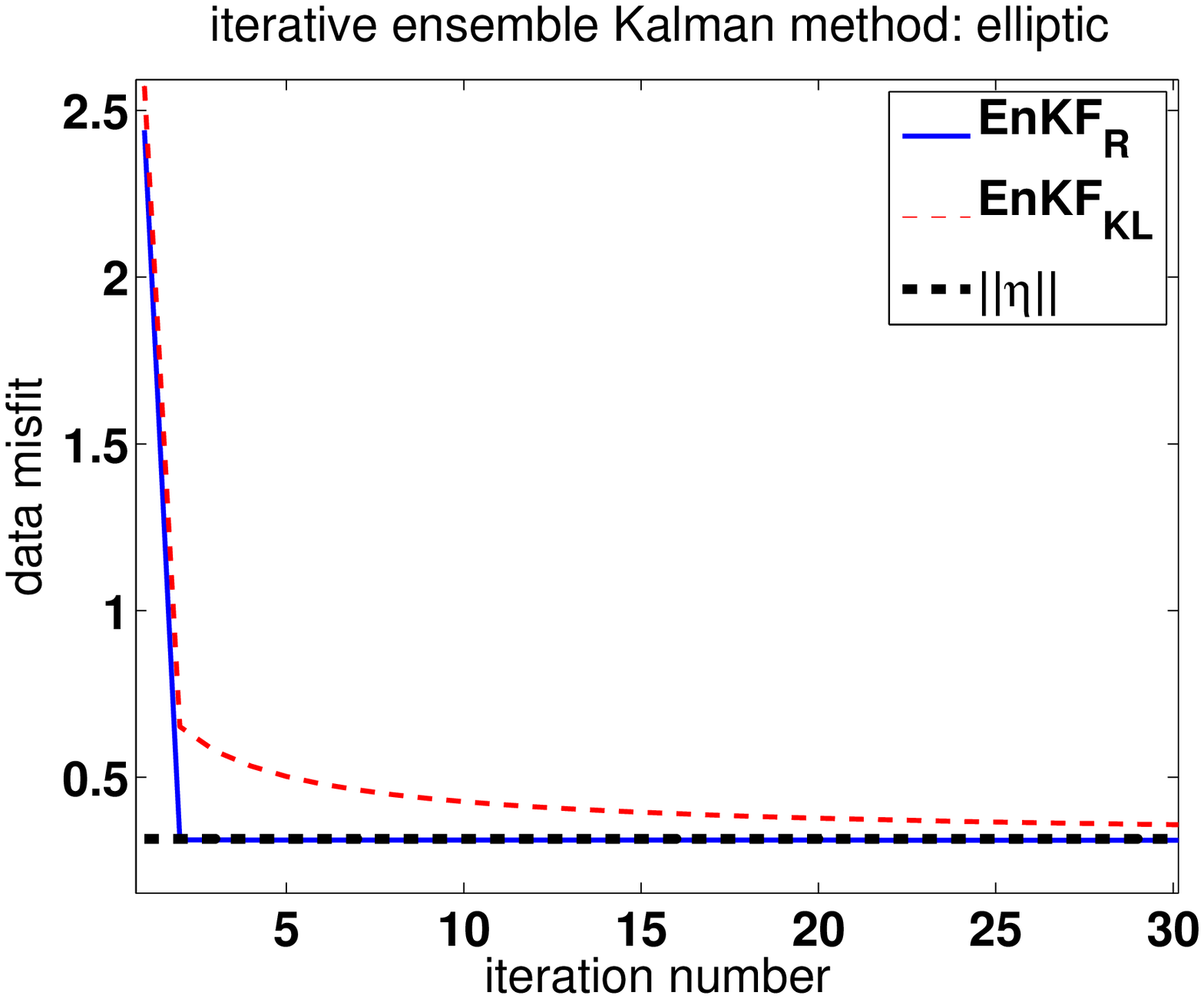}
\end{center}
\caption{Performance of EnKF$_{\rm R}$ and EnKF$_{\rm KL}$.  Left: Relative error with respect to the truth. Right: Data misfit.  }
\label{ped_ensembles}
\end{figure}

\begin{figure}[h]
\begin{center}
       \includegraphics[width=.8\textwidth]{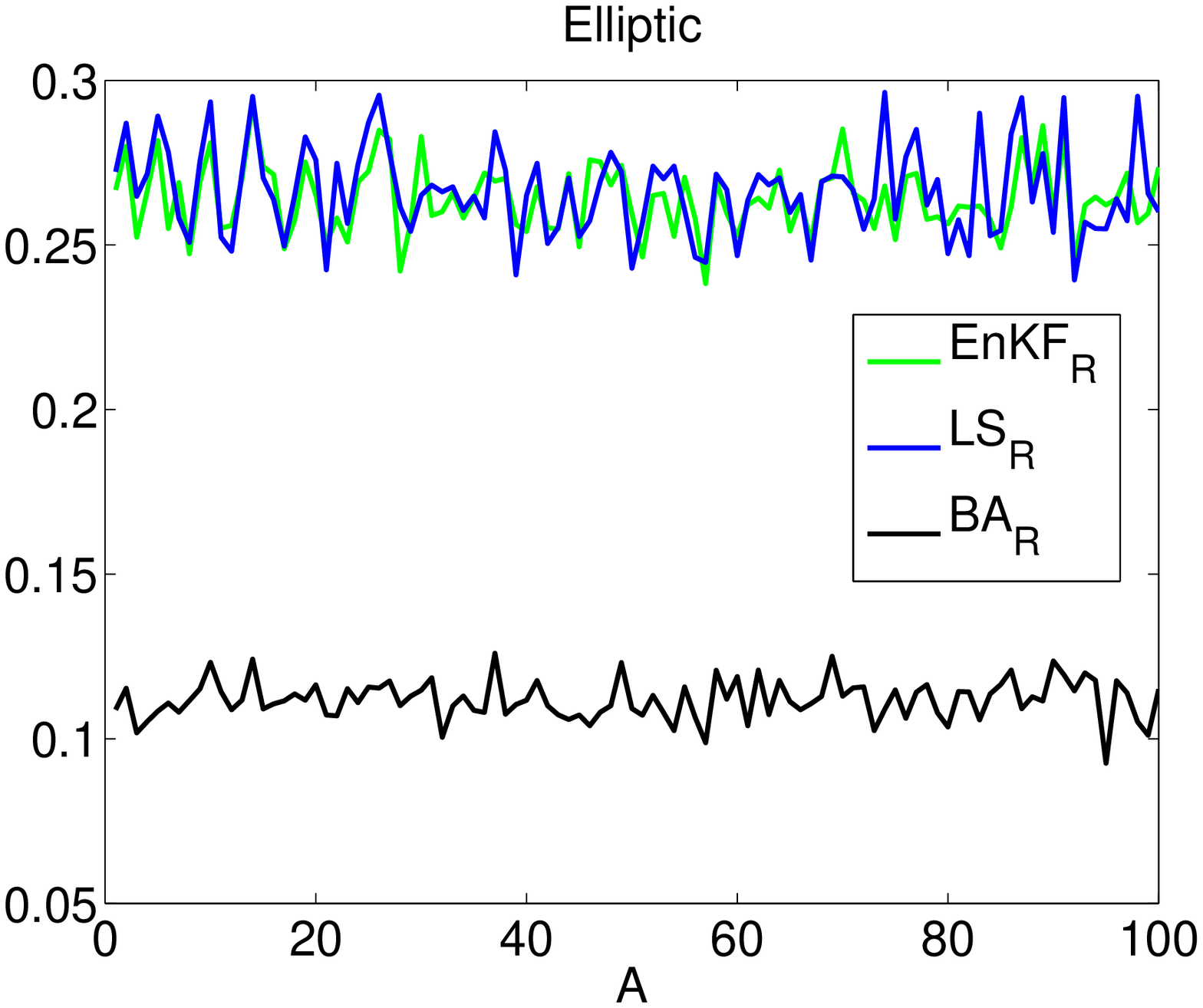}
\end{center}
\caption{Comparison over different subspaces $\cA$, generated from draws from the prior, of the relative errors of one iteration of EnKF$_{\rm R}$ versus
(Tikhonov regularized) least squares (LS$_{\rm R}$) and the best approximation (BA$_{\rm R}$).}
\label{ped_ensembles2}
\end{figure}

\section{Groundwater flow}
\label{sec6}

Next, we will investigate an inverse problem arising in groundwater modeling. Typically, the calibration of subsurface flow models consist of
estimating geologic properties whose model predictions ``best fit''
the measurements of flow-related quantities. In particular, herein we will 
consider the estimation of the conductivity 
of an aquifer form hydraulic head measurements. 
Similarly to the previous section, 
we find comparable performance of EnKF and LS, with the
latter here regularized in a Newton-CG fashion.
Again, BA is included for comparison. 


\subsection{Setting}
We consider groundwater flow in a two-dimensional confined aquifer whose physical domain is $\Omega=[0,6]\times [0,6]$. The hydraulic conductivity is denoted by $K$. The flow in the aquifer, is described in terms of the piezometric head $h(x)$ ($x\in \Omega$) which, in the steady-state is governed by the following equation \cite{Bear}
\begin{eqnarray}\label{sec4:1}
-\nabla\cdot e^{u}\nabla h&=f &\qquad\textrm{in}~~\Omega
\end{eqnarray}
where $u\equiv \log{K}$ and $f$ is defined by 
\begin{eqnarray}\label{sec4:4}
f(x_{1},x_{2})=\left\{\begin{array}{ccc}
0 &\textrm{if}& 0<x_{2}\leq 4,\\
137&\textrm{if}& 4<x_{2}<5,\\
274&\textrm{if}& 5\leq x_{2}<6.\end{array}\right.
\end{eqnarray}
 We consider the following boundary conditions
\begin{eqnarray}\label{sec4:3}
\fl h(x,0)=100, \qquad \frac{\partial h}{\partial x}(6,y)=0,\qquad
-e^{u}\frac{\partial h}{\partial x}(0,y)=500,\qquad   \frac{\partial h}{\partial y}(x,6)=0,
\end{eqnarray}
For the physical interpretation of the source term and boundary
conditions  (\ref{sec4:4})-(\ref{sec4:3}), we refer the reader to
\cite{Carrera} where a similar model was used as a benchmark for
inverse modeling in groundwater flow. A similar model was also studied
in \cite{Hanke, IglesiasDawson1} also in the context of parameter
identification.  We will be interested in the inverse problem of 
recovering the hydraulic conductivity, or more precisely its logarithm
$u$, from noisy pointwise measurements of the piezometric head $h$.
This is a model for the situation in groundwater applications
where observations
of head are used to infer the conductivity of the aquifer. 

\subsection{Numerical Results}

  We let $\cG(u) = \{h(x_k)\}_{k \in \bbK}$
where $\bbK$ is some finite set of points in 
$\Omega$ with cardinality $\kappa$.
In particular, $\bbK$ is given by 
the configuration of $N=100$ observation wells displayed in
Figure \ref{gw:3} (Top right).  
We introduce the prior Gaussian
$\mu_0 \sim \cN(\overline{u},C)$, let $u^\dagger \sim \mu_0$, and
simulate data $y = \cG(u^\dagger) + \eta^\dagger$, 
where $\eta^\dagger \sim \cN(0,\Gamma)$. We let $C = \beta  L^{-\alpha}$ with $L\equiv -\Delta$ defined on $D(L)=\{v\in H^{2}(\Omega)| \nabla v\cdot \mathbf{n}=0, ~\textrm{on}~\partial \Omega, \int_{\Omega}v=0\}$. Additionally, we define $\Gamma=\gamma^2 I$ and we choose $\alpha=1.3, \beta=0.5,$ $\overline{u}=4$ 
and $\gamma=7$ fixed. 
We reiterate from Section \ref{sec:PE} that the space 
$\cA={\rm span} \{ \psi^{(j)}\}_{j=1}^J$
will be chosen based on either draws from the prior $\mu_0$, 
with subscript $R$, or on the \KL basis, with
subscript $KL$.

The forward model (\ref{sec4:1})-(\ref{sec4:3}) is discretized with cell-centered finite differences \cite{Mixed}. For the approximation of the LS problem, we implemented the Newton-CG method of \cite{hanke1997regularizing}. We conduct experiments analogous to those of the previous section and the results are shown in Figures \ref{gw_ensembles}, \ref{gw_ensembles2}, \ref{gw:3} and the second column of Table \ref{Table}. These results are very similar to those shown in the previous section for the
linear elliptic problem, demonstrating the robustness of the
observations made in that section for the solution of inverse problems
in general, using the ensemble Kalman methodology herein.

\begin{figure}[h]
\begin{center}
     \includegraphics[width=.45\textwidth]{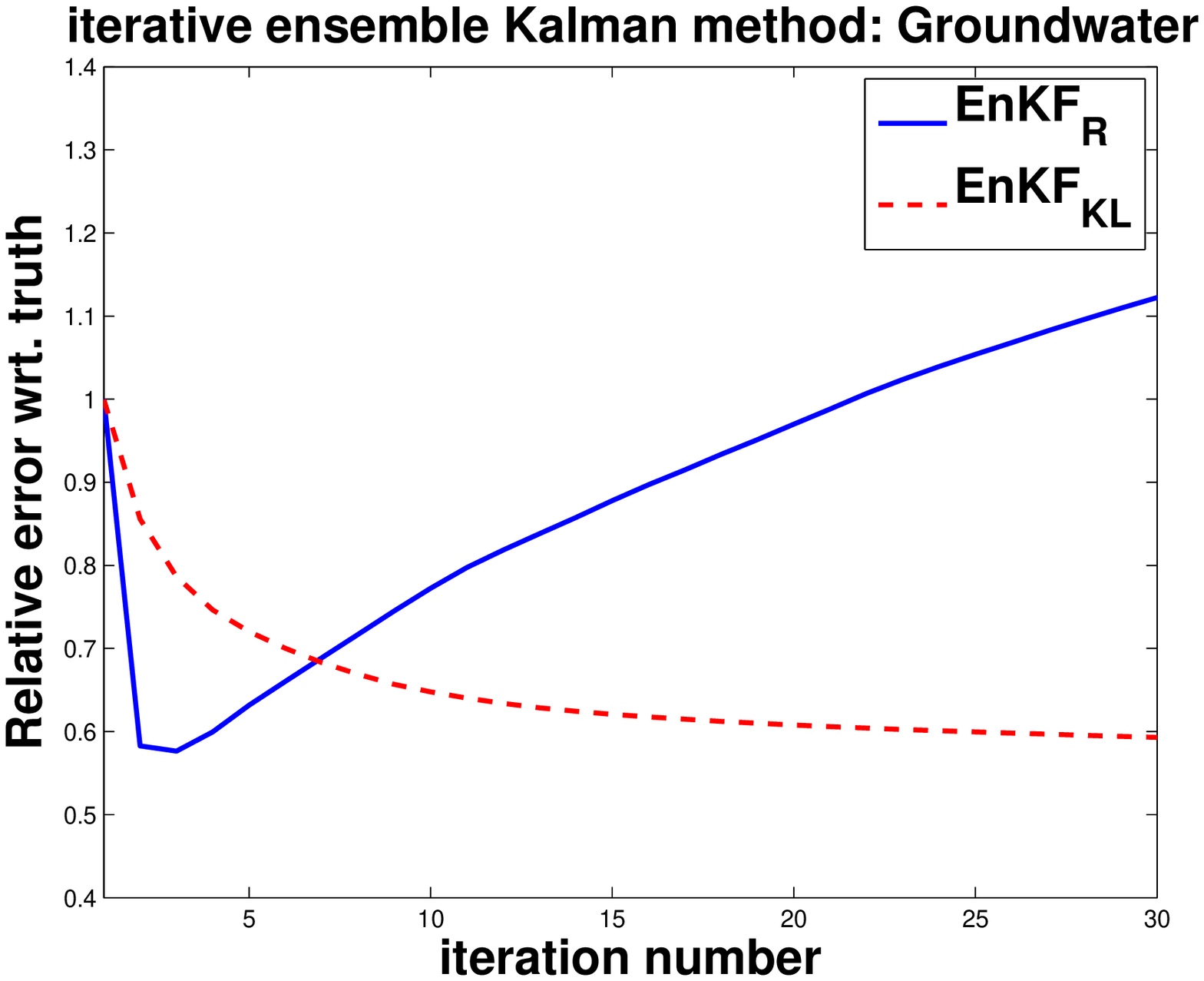}
      \includegraphics[width=.45\textwidth]{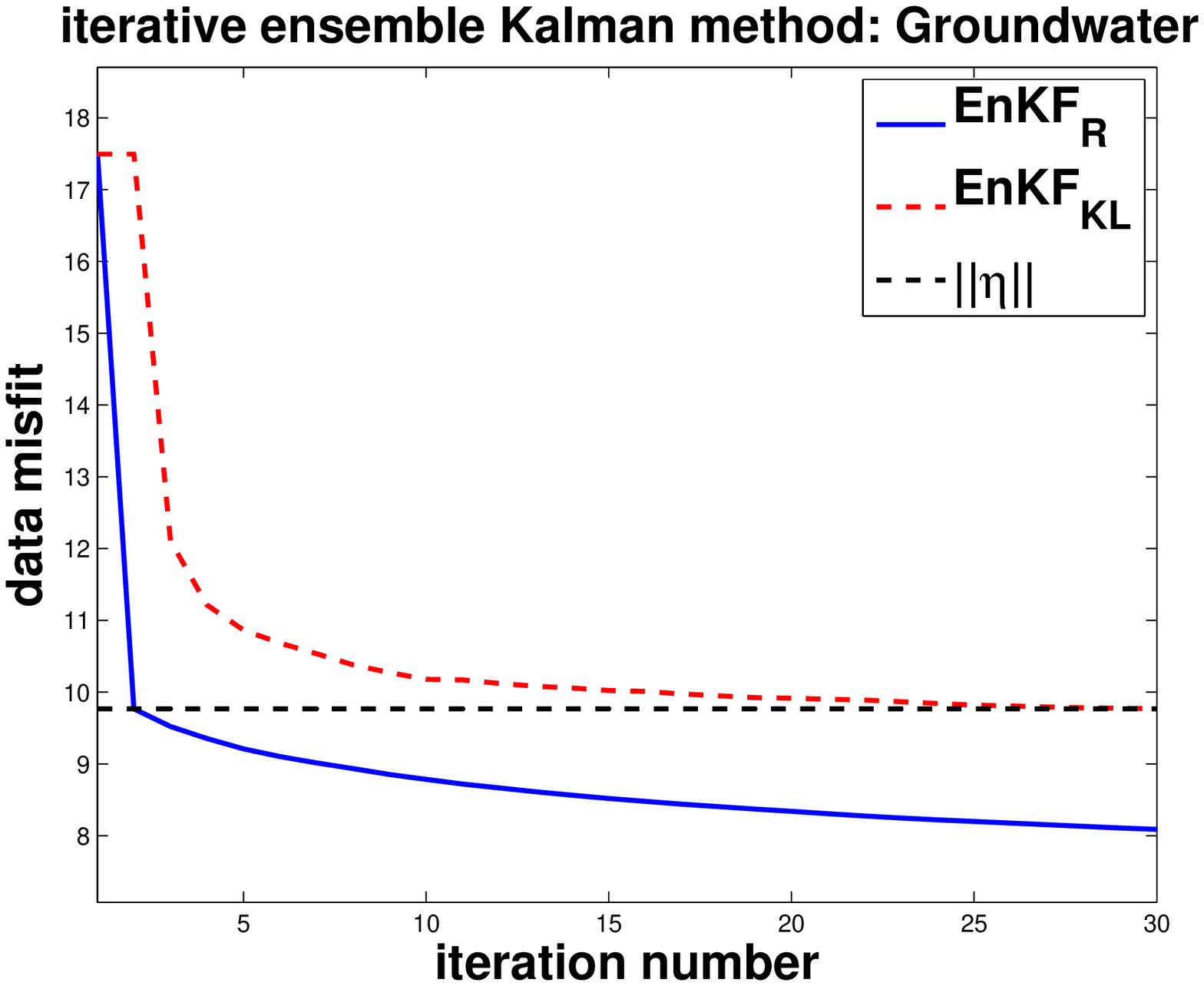}
\end{center}
\caption{Performance of EnKF$_{\rm R}$ and EnKF$_{\rm KL}$. 
Left: Relative error with respect to the truth. Right: Data misfit.  }
\label{gw_ensembles}
\end{figure}

\begin{figure}[h]
\begin{center}
       \includegraphics[width=.8\textwidth]{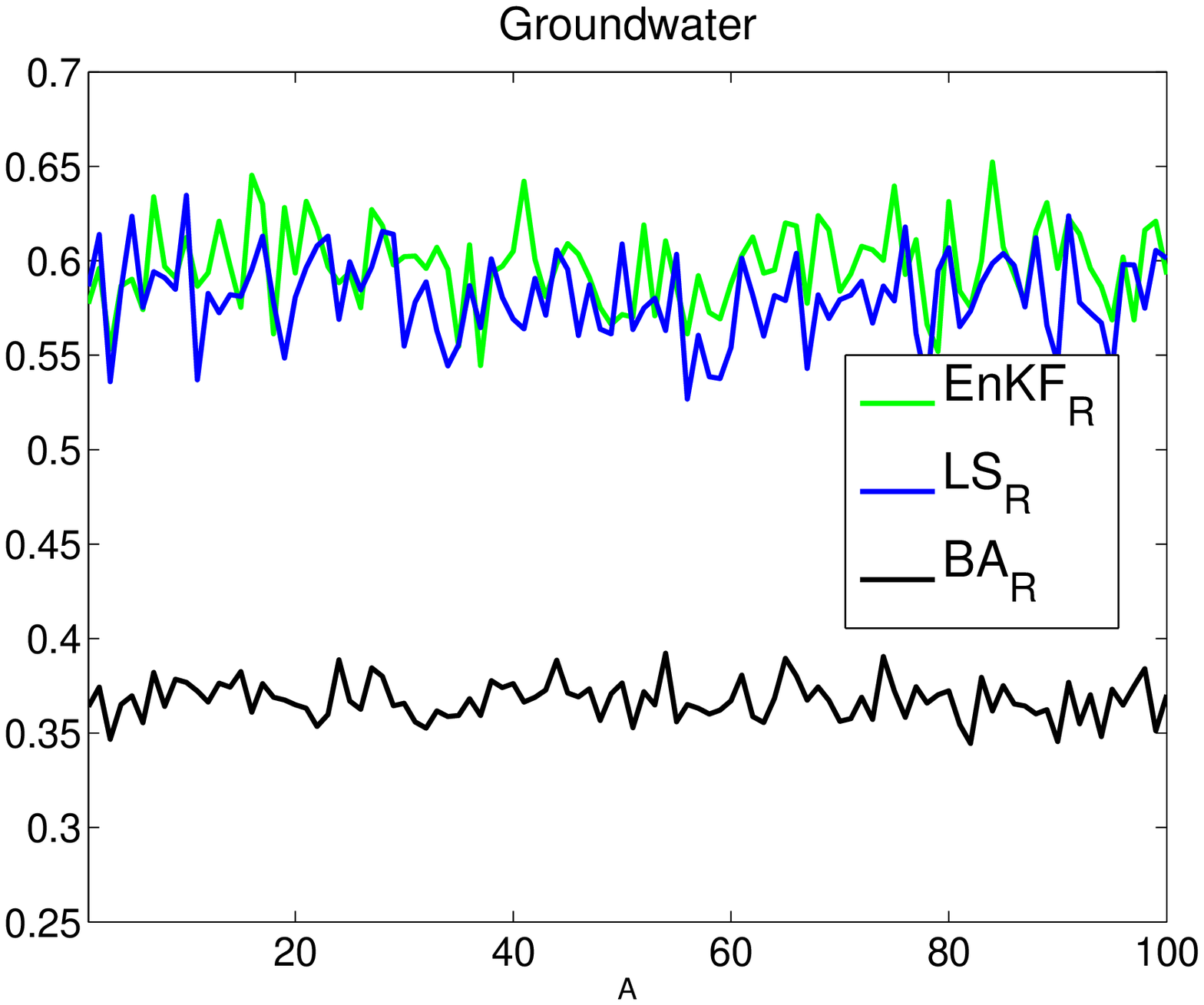}
\end{center}
\caption{Comparison over different subspaces $\cA$, generated from draws from the prior, of the relative errors of one iteration of EnKF$_{\rm R}$ versus
(truncated iteration regularized) least squares (LS$_{\rm R}$) and the
best approximation (BA$_{\rm R}$).} 
\label{gw_ensembles2}
\end{figure}

\begin{figure*}
       \includegraphics[width=0.24\textwidth]{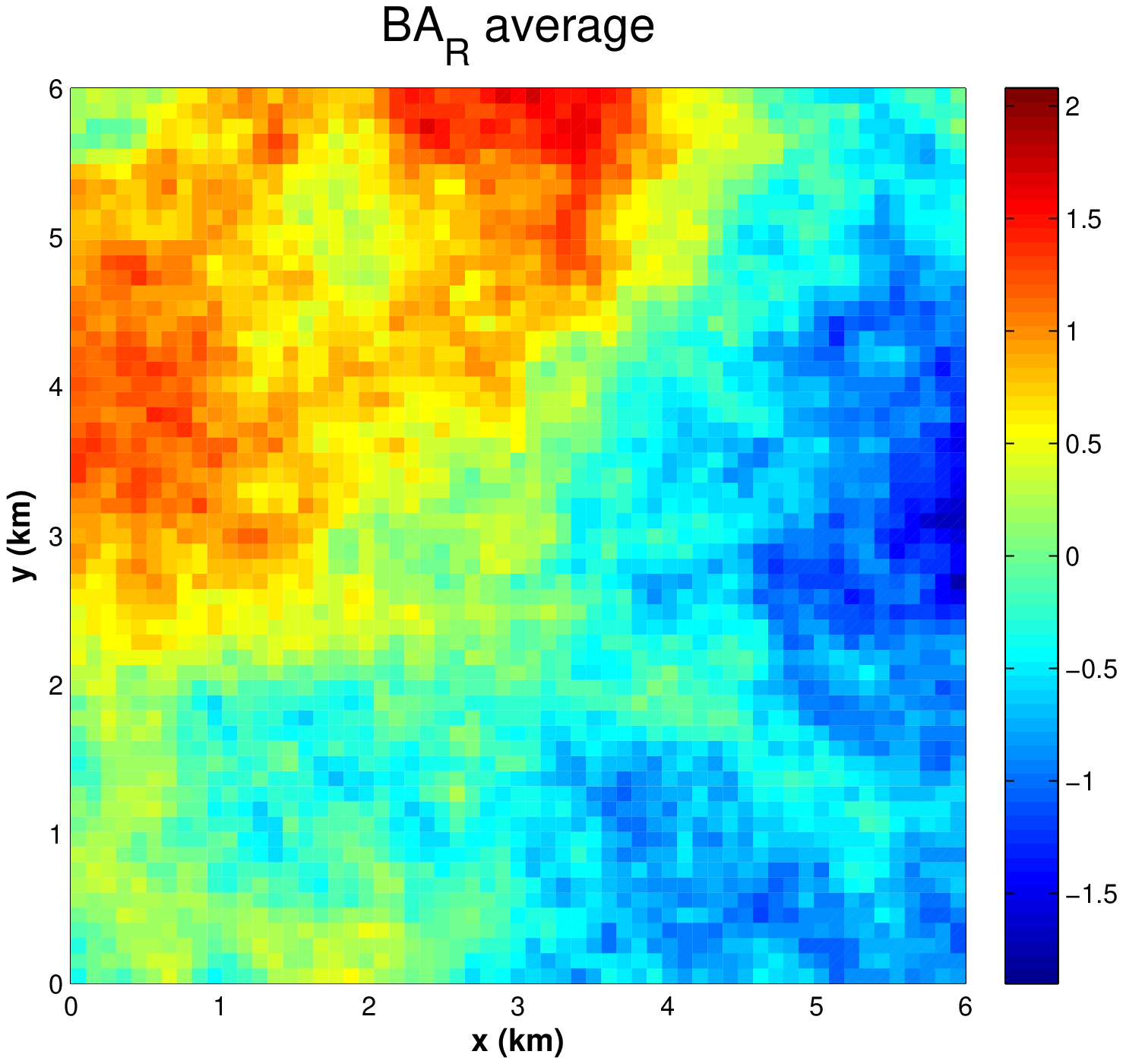}
	 \includegraphics[width=0.24\textwidth]{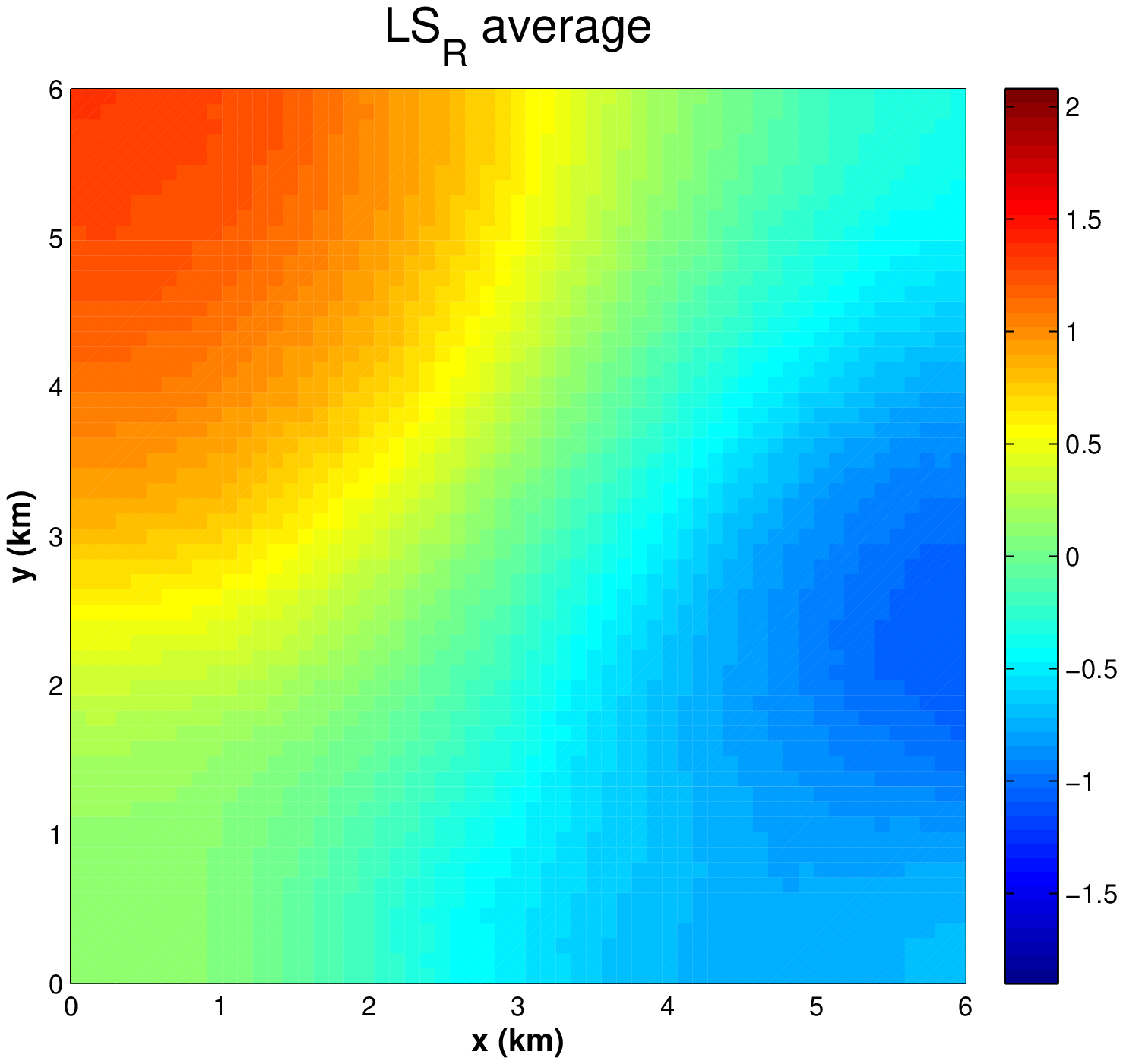}
	 \includegraphics[width=0.24\textwidth]{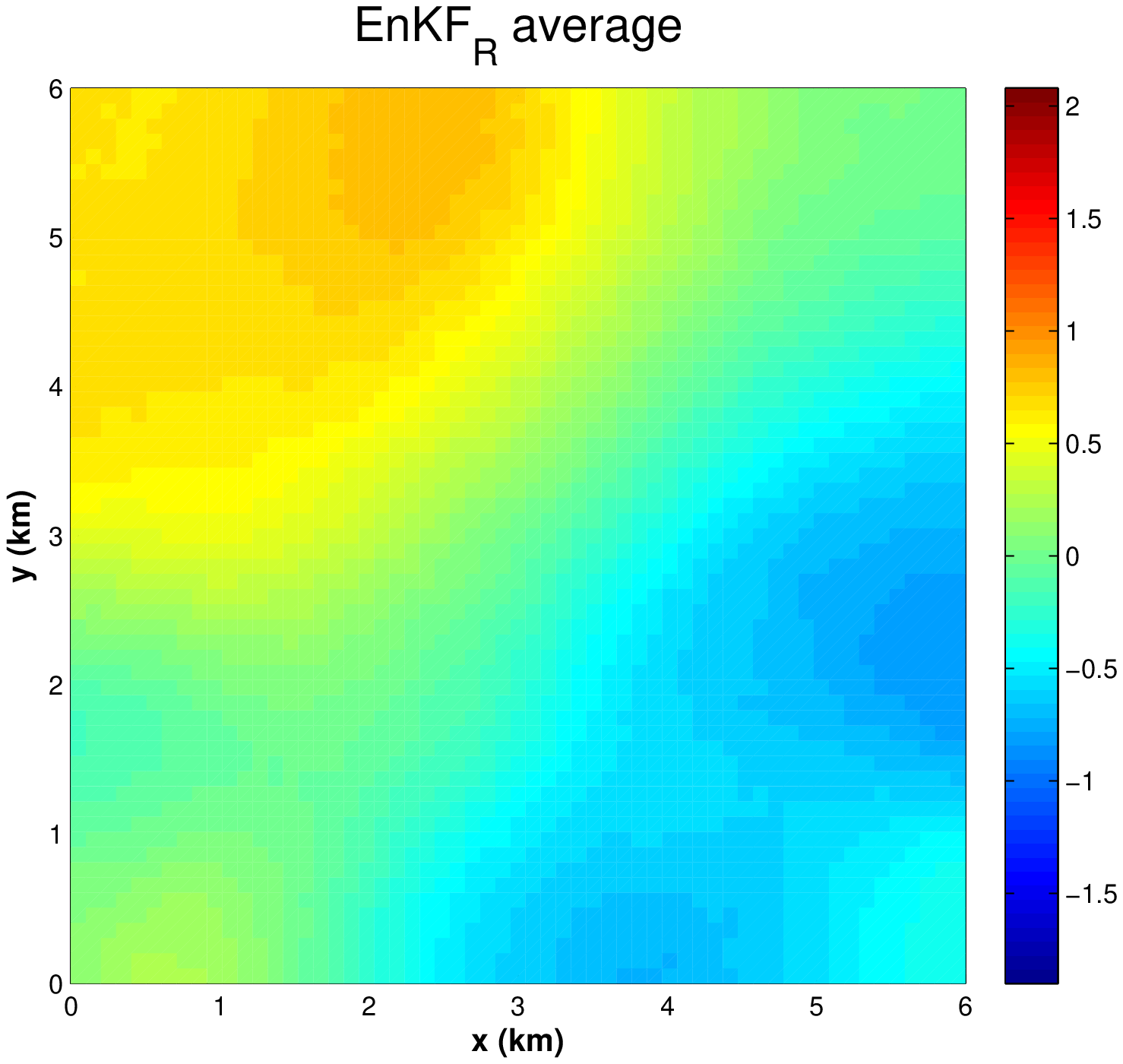}
	 \includegraphics[width=0.24\textwidth]{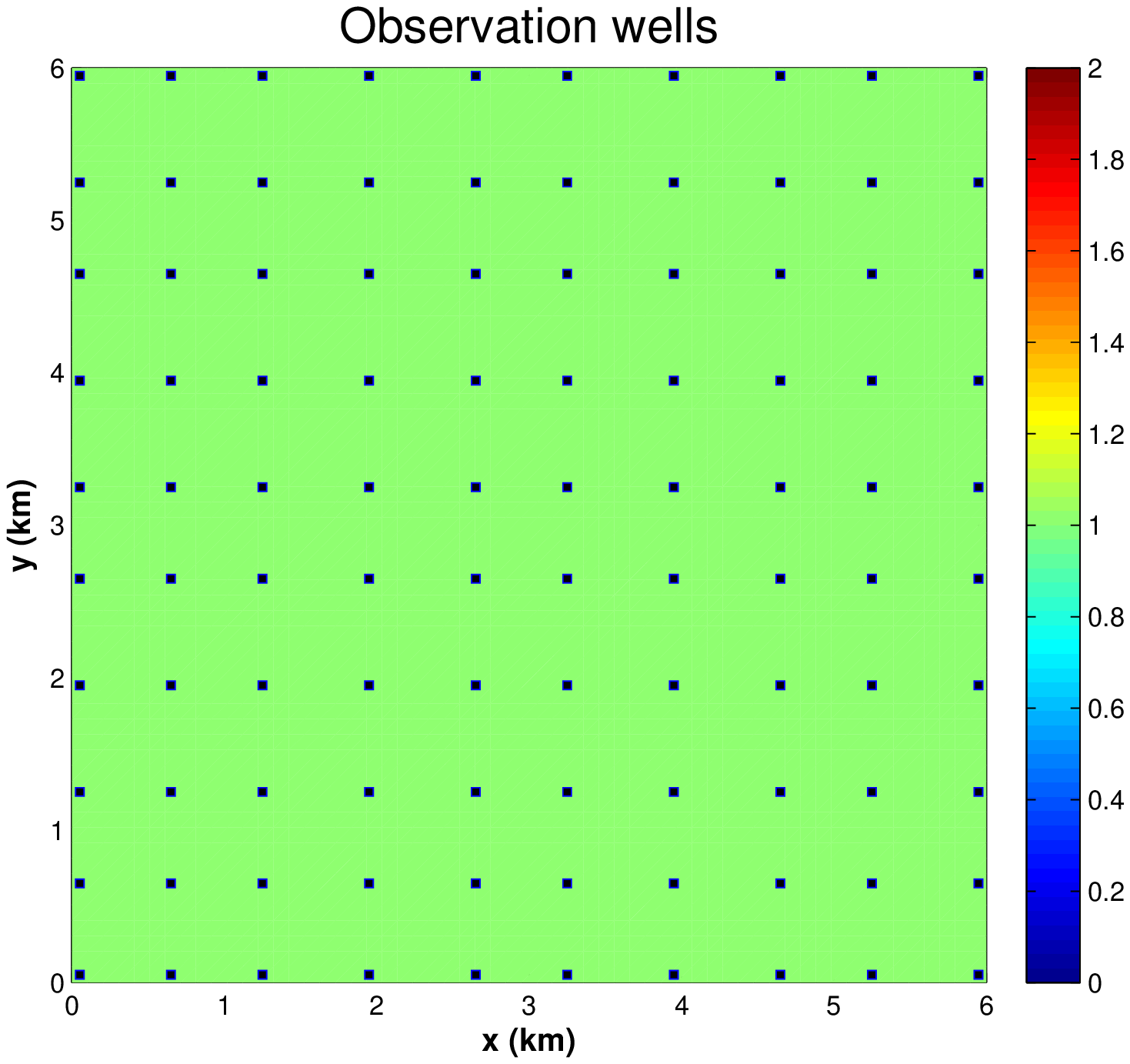}\\
\includegraphics[width=0.24\textwidth]{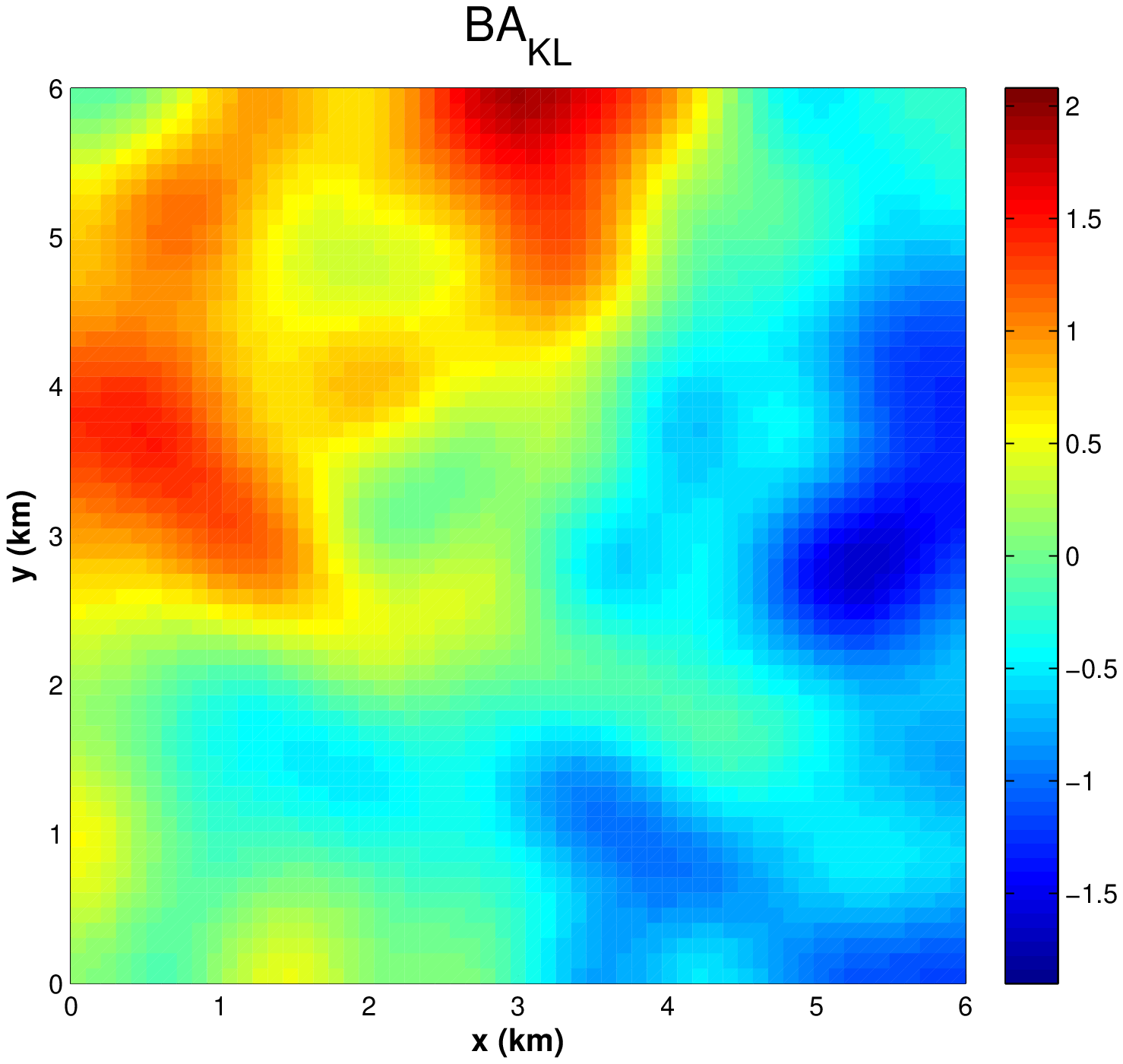}       
\includegraphics[width=0.24\textwidth]{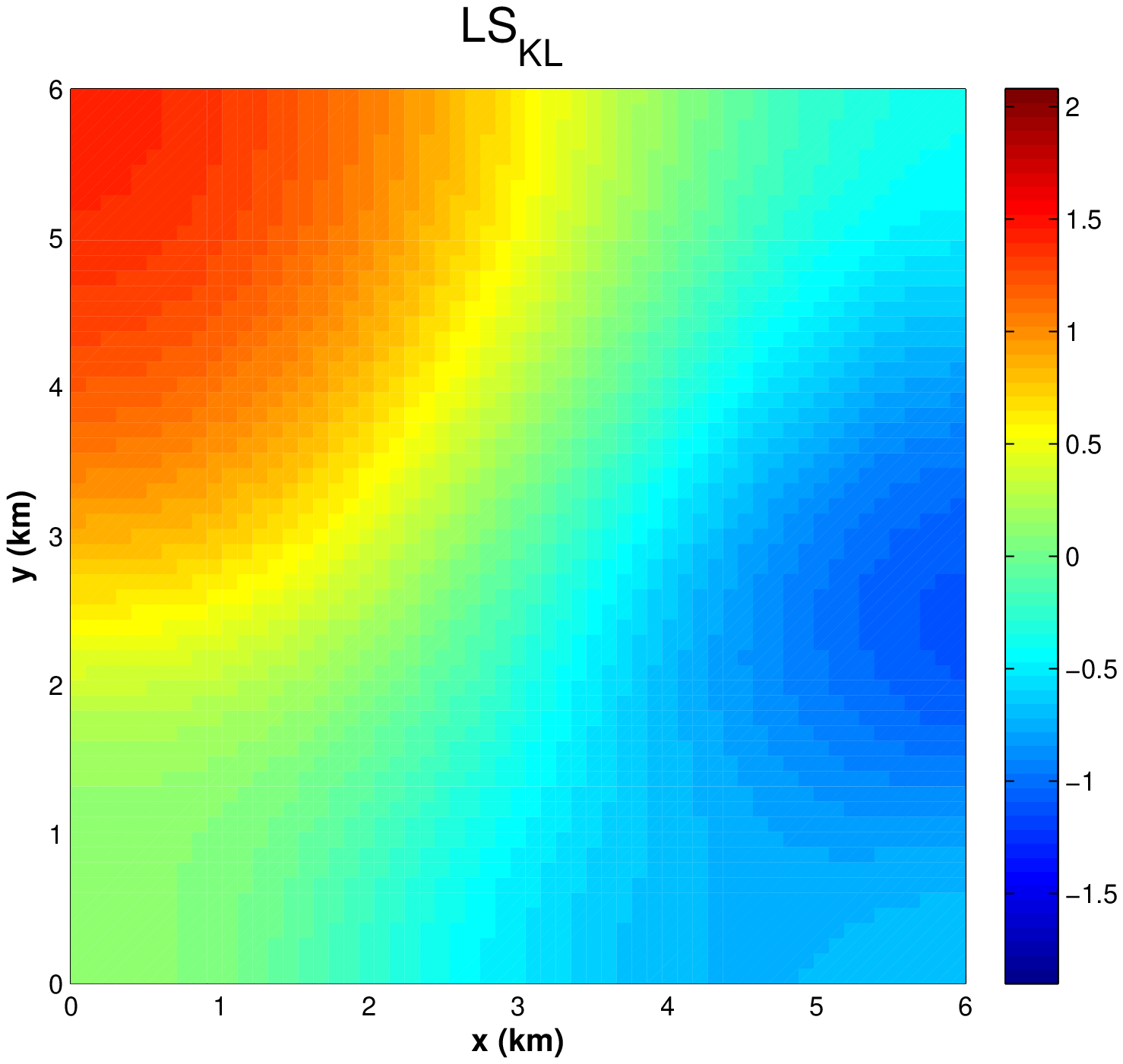}
       \includegraphics[width=0.24\textwidth]{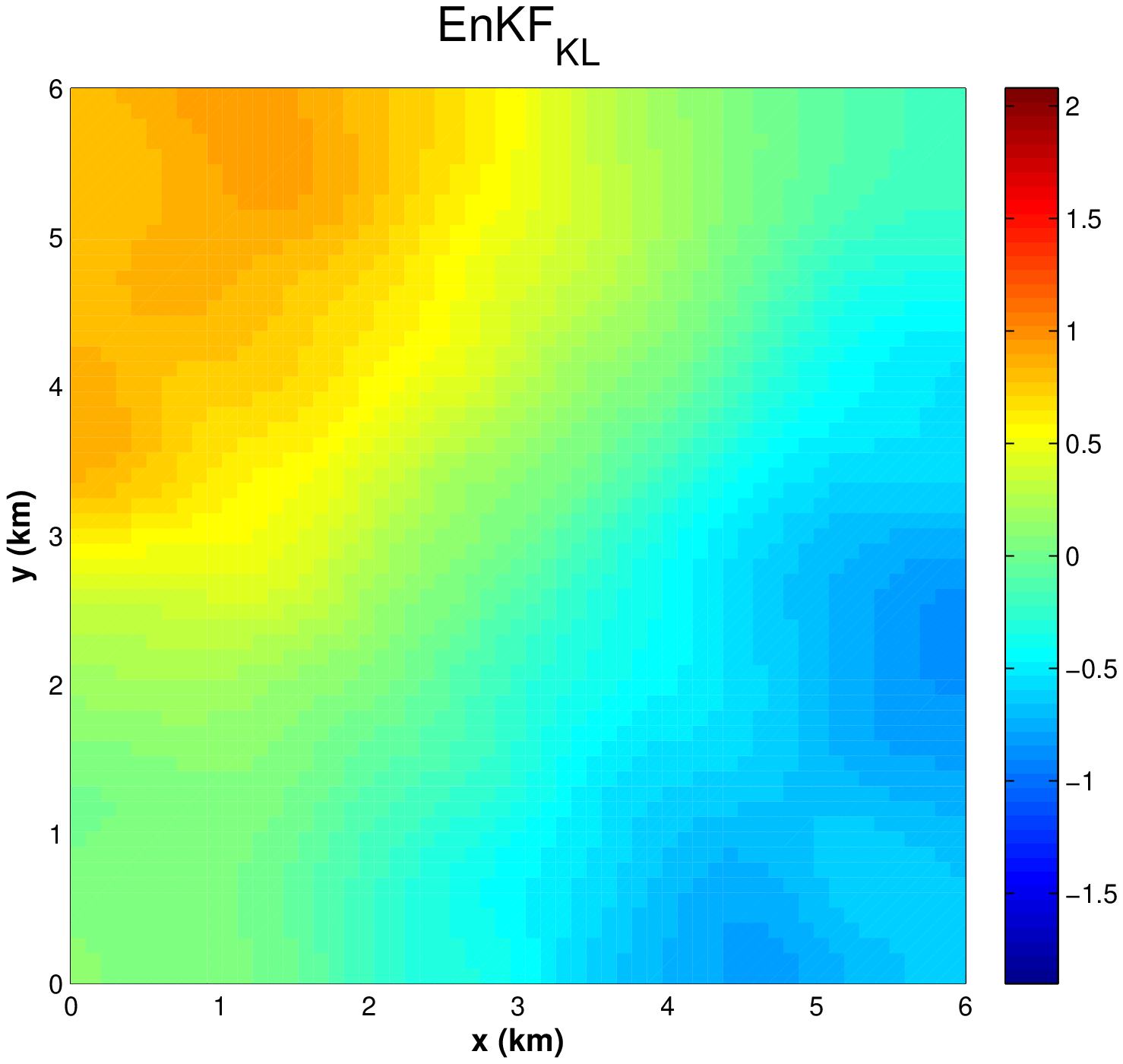} 
       \includegraphics[width=0.24\textwidth]{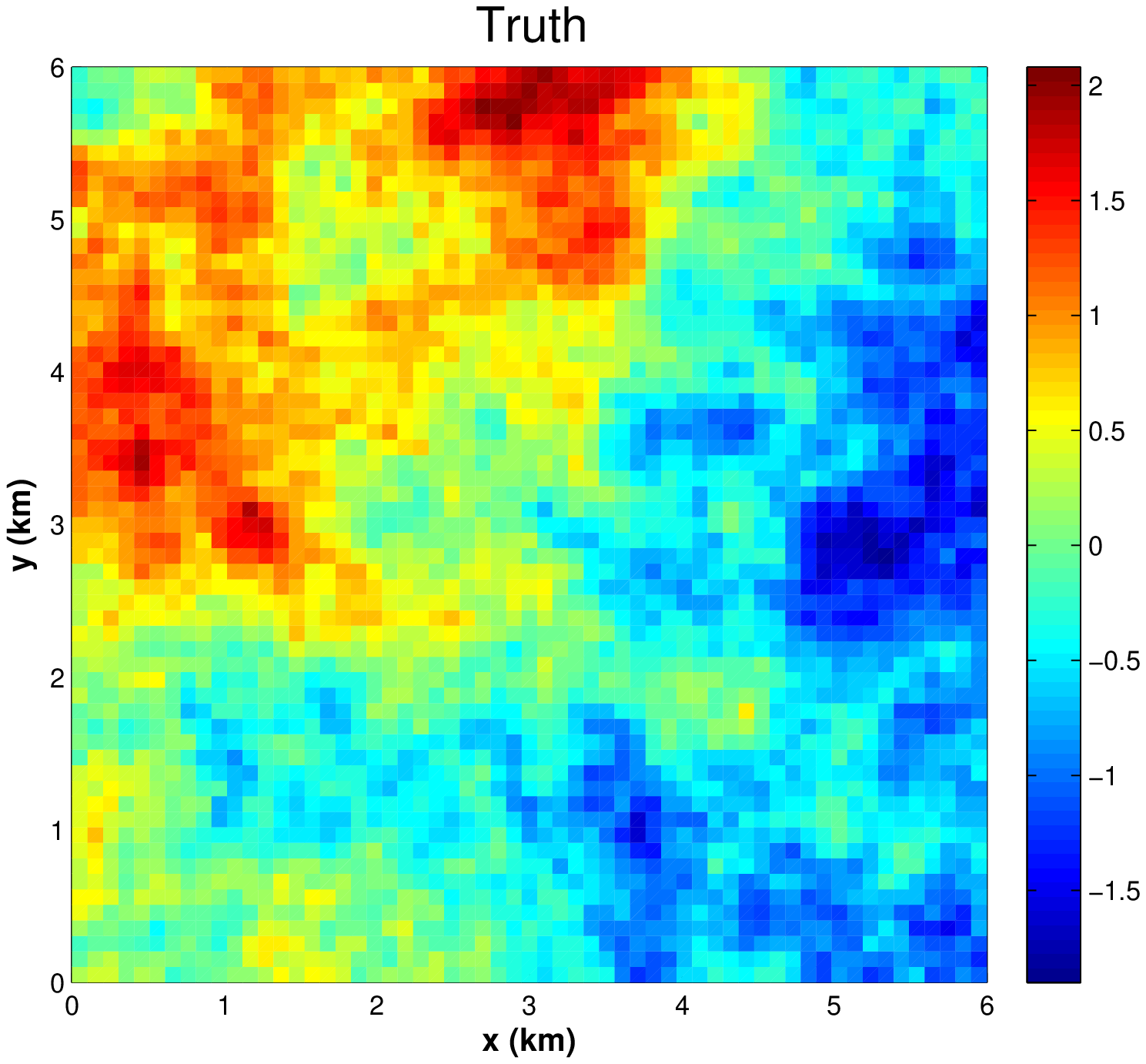}    
\caption{Estimated log $K$.
 Top left: BA$_{\rm R}$, average.
Top middle left: LS$_{\rm R}$, average.
Top middle right: first iteration of EnKF$_{\rm R}$, average.
Top right: measurement well locations. 
Bottom left: BA$_{\rm KL}$. 
Bottom middle left: LS$_{\rm KL}$.  
Bottom middle right: EnKF$_{\rm KL}$.
Bottom right: the truth $u^\dagger$. }\label{gw:3}
\end{figure*}

\section{Navier-Stokes Equation}\label{sec7}

In this section, we consider an inverse problem in
fluid dynamics, which is relevant to data assimilation applications in
oceanography and meteorology.   
In particular, we examine the problem of 
recovering the initial condition of the Navier-Stokes Equation 
(NSE), given noisy pointwise observations of the velocity
field at later times.
We will
investigate a regime in which the combination of viscosity, time-interval,
and truncation of the forward model is such that the exponential ill-posedness of the inverse problem is alleviated.

\subsection{Setting}

We consider the 2D Navier-Stokes
equation on the torus $\bbT^{2} := [-1,1) \times [-1,1)$ with periodic
boundary conditions:
\begin{eqnarray*}
  \begin{array}{cccc}
    \partial_{t}v - \nu \Delta v + v \cdot \nabla v + \nabla p &=& f
    & {\rm for~ all ~}
    (x, t) \in \bbT^{2} \times (0, \infty), \\
    \nabla \cdot v &=& 0 &{\rm for~ all~ }
    (x, t) \in \bbT^{2} \times (0, \infty), \\
    v &=& u &{\rm for~ all ~}
    (x,t) \in \bbT^{2} \times \{0\}.
  \end{array}
  \label{eq:NSE}
\end{eqnarray*}
Here $v \colon \bbT^{2} \times (0, \infty) \to \R^{2}$ is a
time-dependent vector field representing the velocity, $p \colon
\bbT^{2} \times (0,\infty) \to \R$ is a time-dependent scalar field
representing the pressure, $f \colon \bbT^{2} \to \R^{2}$ is a vector
field representing the forcing (which we assume to be time-independent
for simplicity), and $\nu$ is the viscosity.  
We are interested in the inverse problem of
determining the initial velocity field $u$ from
pointwise measurements of the velocity
field at later times.
This is a model for the situation in weather forecasting
where observations of the atmosphere are used to improve
the initial condition used for forecasting.

\subsection{Numerical Results}

We let $t_j=jh$, for $j=1, \dots, J$,
and $\cG(u) = \{v(x_k,t_j)\}_{(j,k) \in \bbK'}$
where $\bbK'=\{1,\cdots,J\} \times \bbK$ and
$\bbK$ is some finite set of points in 
$\Omega$ with cardinality $\kappa$.  In particular, 
we take $\bbK$ to be the set of grid points
in physical space implied by the underlying spectral
truncation used in the numerical integration
(details below). As discussed in Section \ref{sec:PE}, 
we introduce a prior
$\mu_0$, let $u^\dagger \sim \mu_0$, and
simulate data $y = \cG(u^\dagger) + \eta^\dagger$, 
where $\eta^\dagger \sim \cN(0,\Gamma)$. 
As in the previous section, we let $\Gamma=\gamma^2 I$. We fix $\gamma=0.01$ for all
our experiments. The prior $\mu_0$ 
is defined to be the empirical measure supported
on the attractor, i.e. it is 
defined by samples of a trajectory of the forward 
model after convergence to statistical
equilibrium. We use $10^4$ time-steps to construct this
empirical measure. Once again the space 
$\cA={\rm span} \{ \psi^{(j)}\}_{j=1}^J$
is chosen based on $J$ samples from
the (now non-Gaussian) prior $\mu_0$, 
or on a \KL expansion based on 
the empirical mean and variance
from the long forward simulation giving rise to $\mu_0.$
It is important to note that the truth $u^{\dagger}$
is included in the long trajectory used to construct $\mu_0$.
Therefore, some of the initial ensembles drawn from $\mu_0$ 
end up containing a snapshot very close to the truth; such
results are overly optimistic as this is not the typical
situation. It is interesting that all the methods we
study here have comparably overly optimistic results for such ensembles.
scenario of this section, we use the the empirical mean $\bar{u}$ 

The forcing in $f$ is taken to be
$f=\nabla^{\perp}\psi$, where $\psi=\cos(\pi k \cdot x)$ and
$\nabla^{\perp}=J\nabla$ with $J$ the canonical skew-symmetric matrix,
and $k=(5,5)$.  The method used to approximate the forward model
is a modification of a fourth-order Runge-Kutta method,
ETD4RK \cite{cox2002exponential}, in which the Stokes semi-group is
computed exactly by working in the incompressible Fourier basis
$\{\psi_{k}(x)\}_{k \in {\mathbb Z}^2\backslash\{0\}}$, and Duhamel's
principle (variation of constants formula) is used to incorporate the
nonlinear term.  We use a time-step of $dt = 0.005$.
Spatially, a Galerkin spectral method
\cite{hesthaven2007spectral} is used, in the same basis,
and the convolutions arising from products in the nonlinear term are
computed via FFTs.  
We use a double-sized domain in each dimension,
padded with zeros, resulting in $64^2$ grid-point FFTs, and only
half the modes in each direction are retained when transforming back
into spectral space again.  This prevents aliasing, which is avoided as
long as more than one third of the domain in which FFTs are computed 
consists of such padding with zeros.  
The dimension of
the attractor is determined by the viscosity parameter $\nu$.  For the
particular forcing used there is an explicit steady state for all
$\nu>0$ and for $\nu \geq 0.035$ this solution is stable (see
\cite{majda2006non}, Chapter 2 for details).  As $\nu$ decreases the
flow becomes increasingly complex and we focus 
subsequent studies of the inverse problem 
on the mildly chaotic regime which arises
for  $\nu = 0.01$. 
Regarding observations, we let $h = 4 \times dt =
0.02$ and take $J=10$, so that $T=0.2$.   
The observations are made at all numerically 
resolved, and hence observable, wavenumbers in the system;
hence $K=32^2$, because of the padding to avoid
aliasing. 

The numerical results resulting from
these experiments are displayed in
in Figures \ref{Fign1}, \ref{Fign1B}, \ref{Fign2} and the third column of Table \ref{Table}.
The results are very similar to those of the previous two
sections, qualitatively: EnKF and LS 
type methods perform comparably, in both the case of random and \KL
based initial draws; furthermore the lower bound 
produced by BA type methods is of similar order of magnitude
to the EnKF-based methods although the actual error is, of
course, smaller.  However, we also see that the behavior 
of the iterated EnKF$_{\rm R}$ is quite different 
from what we observed in the previous sections since it
decreases monotonically. We conjecture
that this is because of the mildly ill-posed nature of this
problem. Indeed we have repeated the results (not shown) of this section at higher viscosity $\nu=0.1$, where linear damping in the
forward model induces greater ill-posedness, and confirmed
that we recover results for the iterated
EnKF$_{\rm R}$ which are similar to those in the previous sections. Finally, we have checked that the behavior of the error
in the iterated EnKF$_{\rm R}$ is repeatable for Gaussian prior
$\mu_0$, and hence is not a result of non-Gaussian
ensembles used in the figures.

\begin{figure}[h]
\begin{center}
     \includegraphics[width=.45\textwidth]{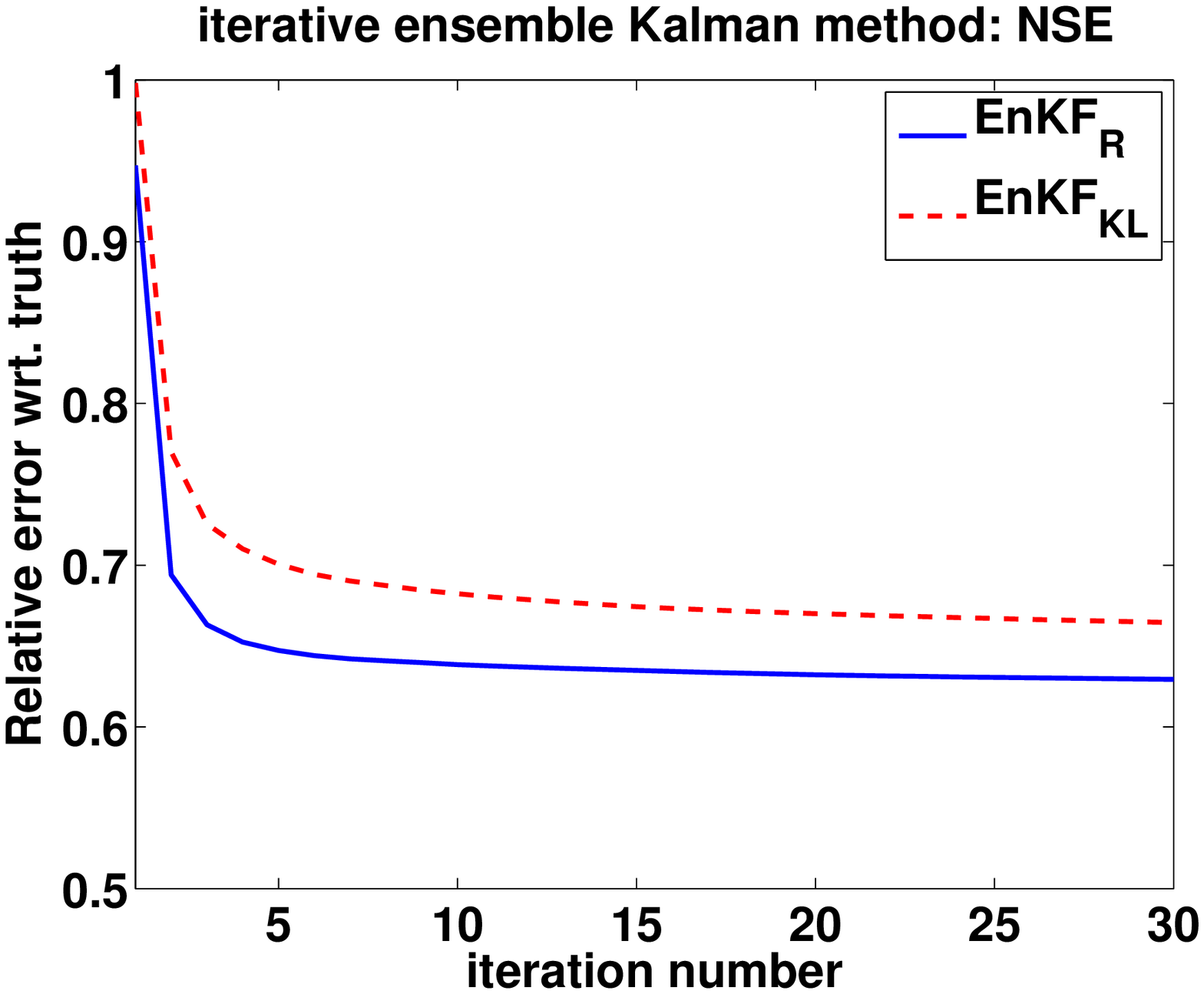}
     \includegraphics[width=.45\textwidth]{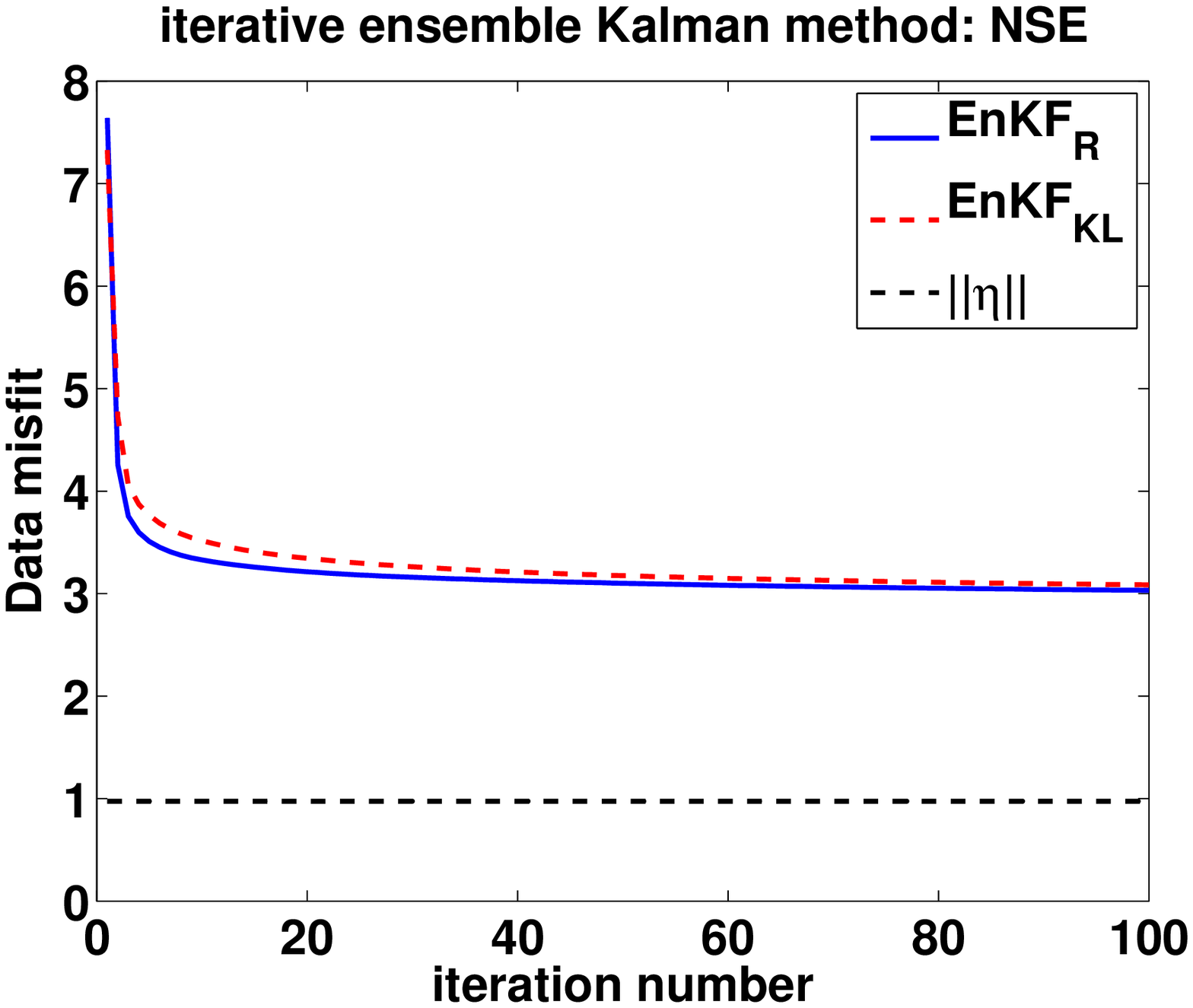}
\end{center}
\caption{Performance of EnKF$_{\rm R}$ and EnKF$_{\rm KL}$. Left: Relative error with respect to the truth. Right: Data misfit.}
\label{Fign1}
\end{figure}

\begin{figure}[h]
\begin{center}
       \includegraphics[width=.8\textwidth]{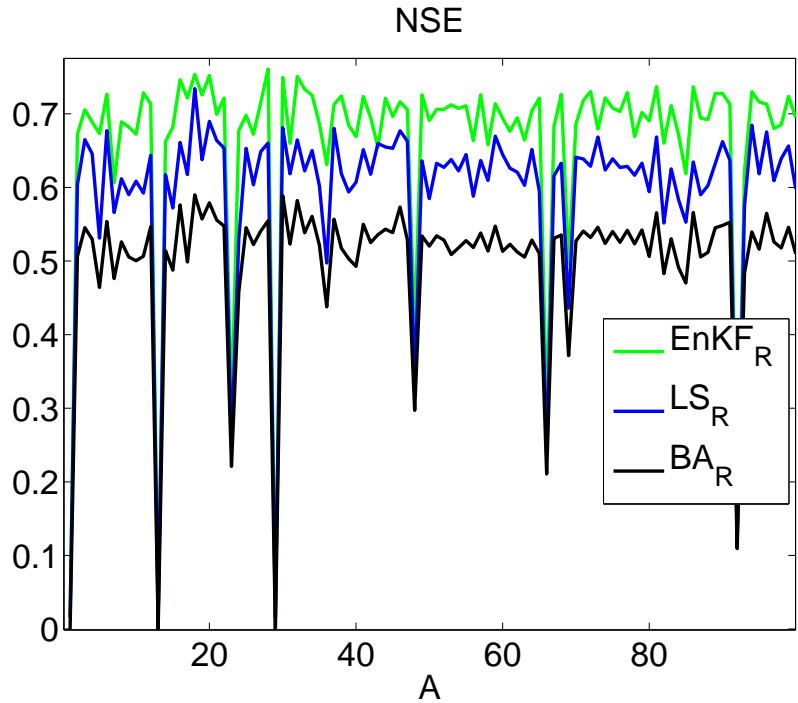}
\end{center}
\caption{Comparison over an ensemble of random ensembles $\cA$
of the relative errors of one iteration of EnKF$_{\rm R}$ versus
least squares (LS$_{\rm R}$) and the best approximation (BA$_{\rm R}$). 
}
\label{Fign1B}
\end{figure}

\begin{figure*}
       \includegraphics[width=0.24\textwidth]{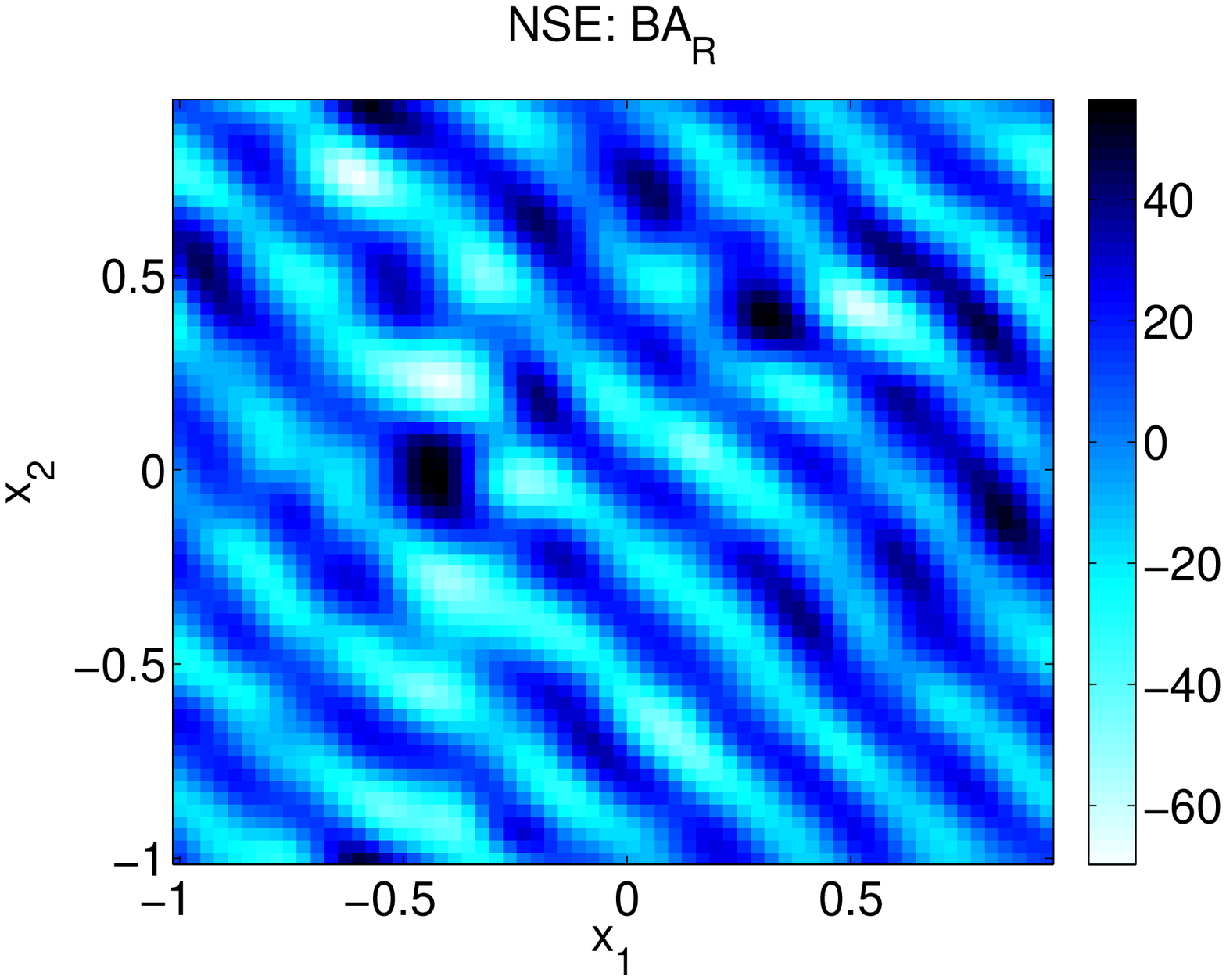}
	 \includegraphics[width=0.24\textwidth]{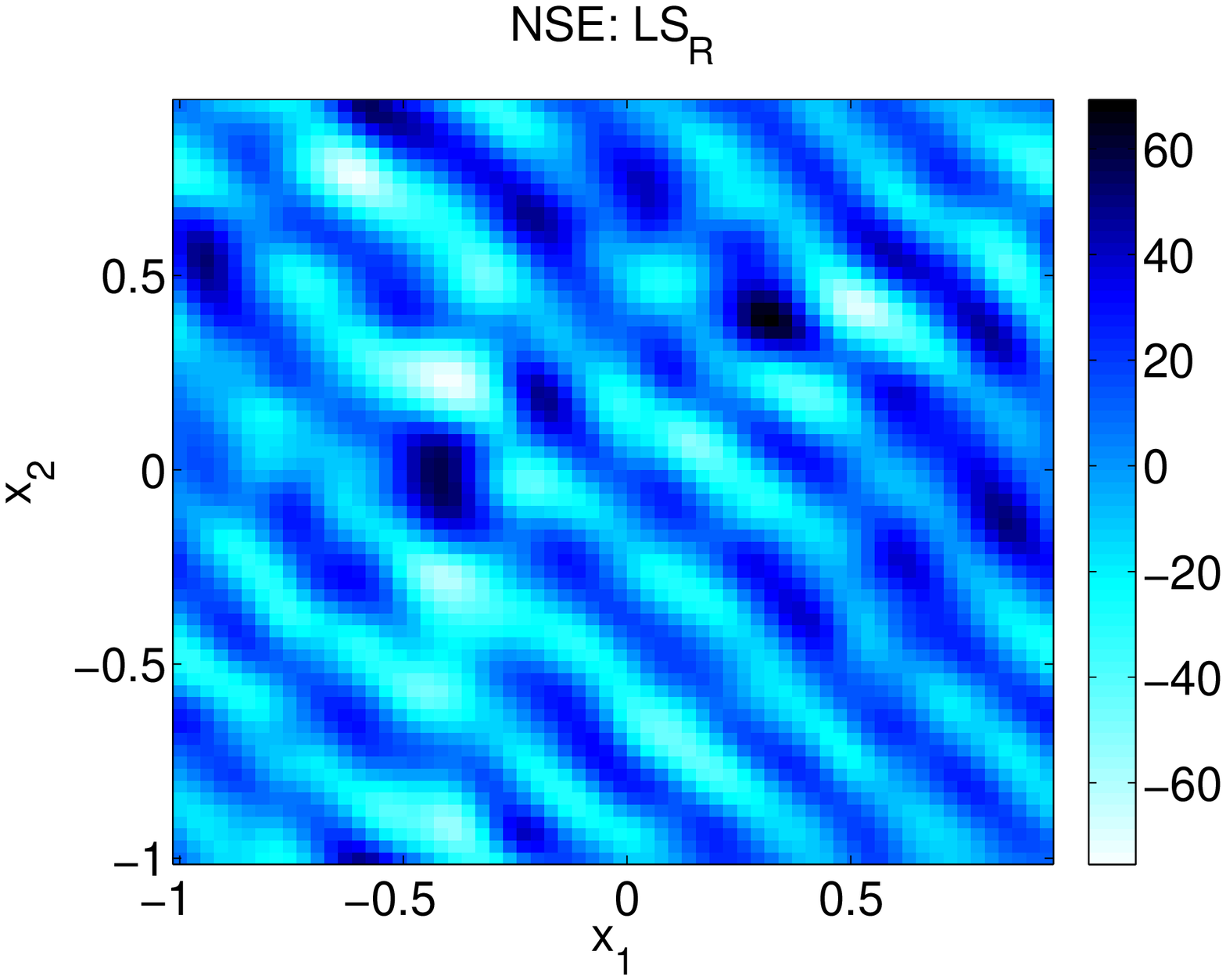}
	 \includegraphics[width=0.24\textwidth]{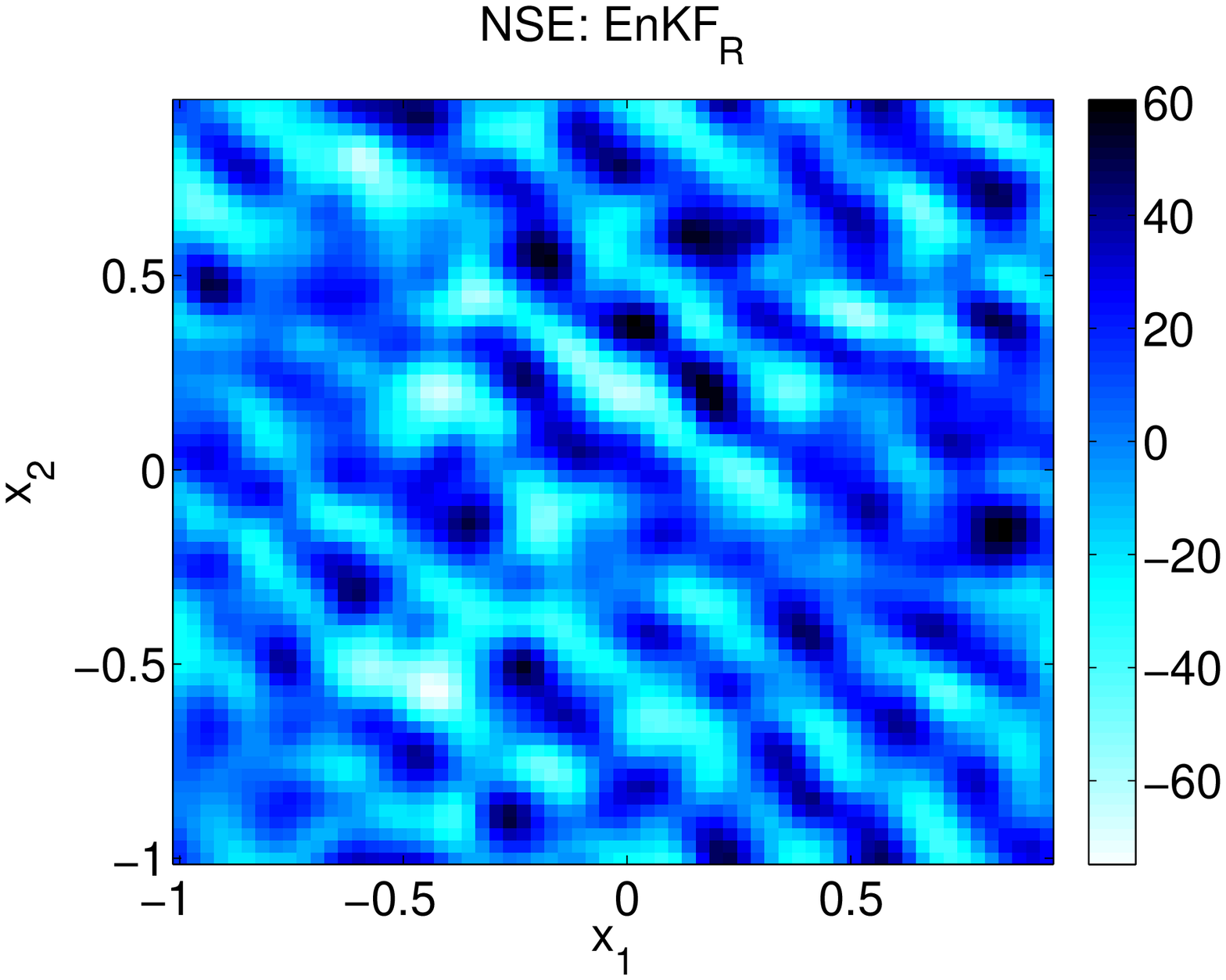}
	 \includegraphics[width=0.24\textwidth]{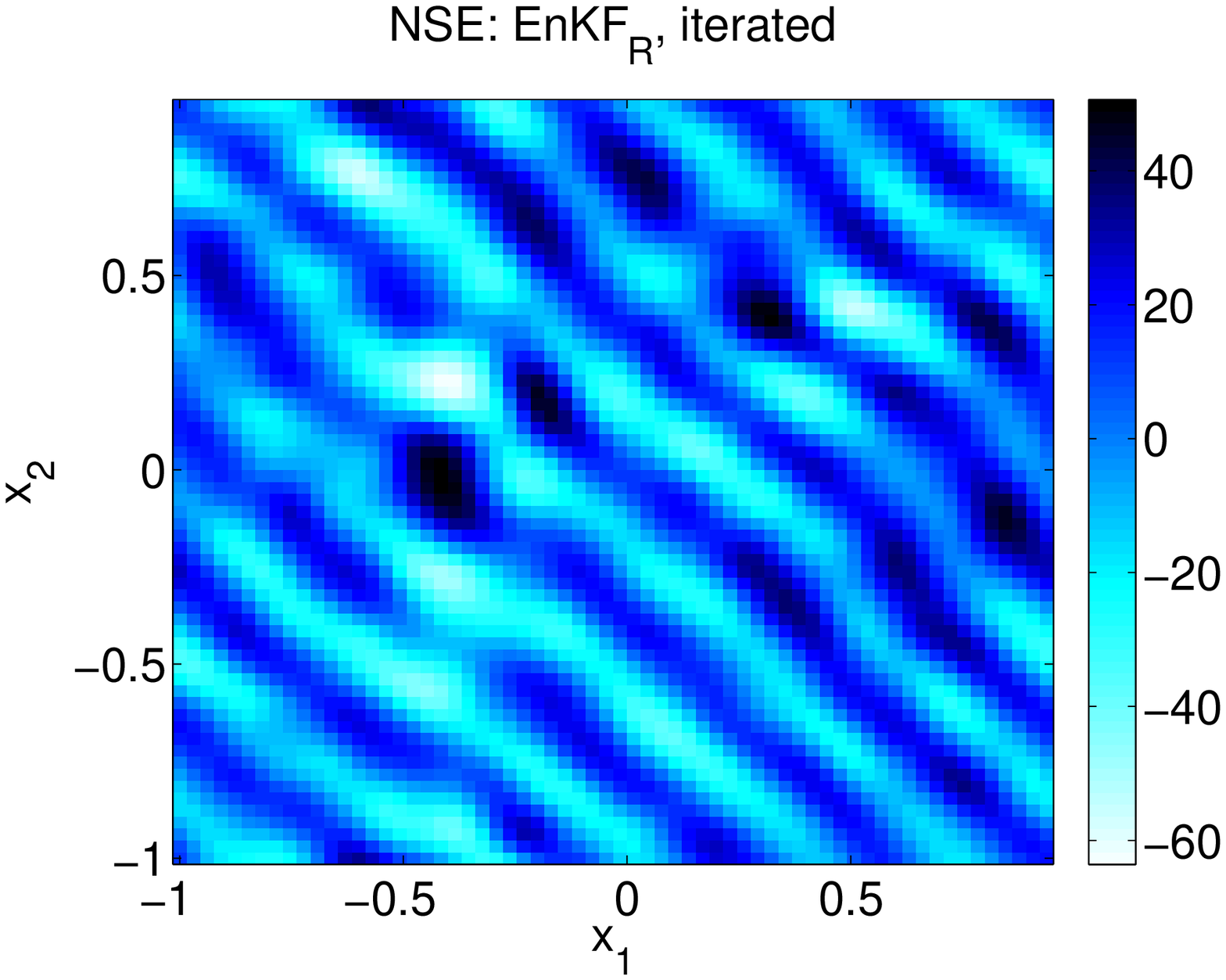}\\
\includegraphics[width=0.24\textwidth]{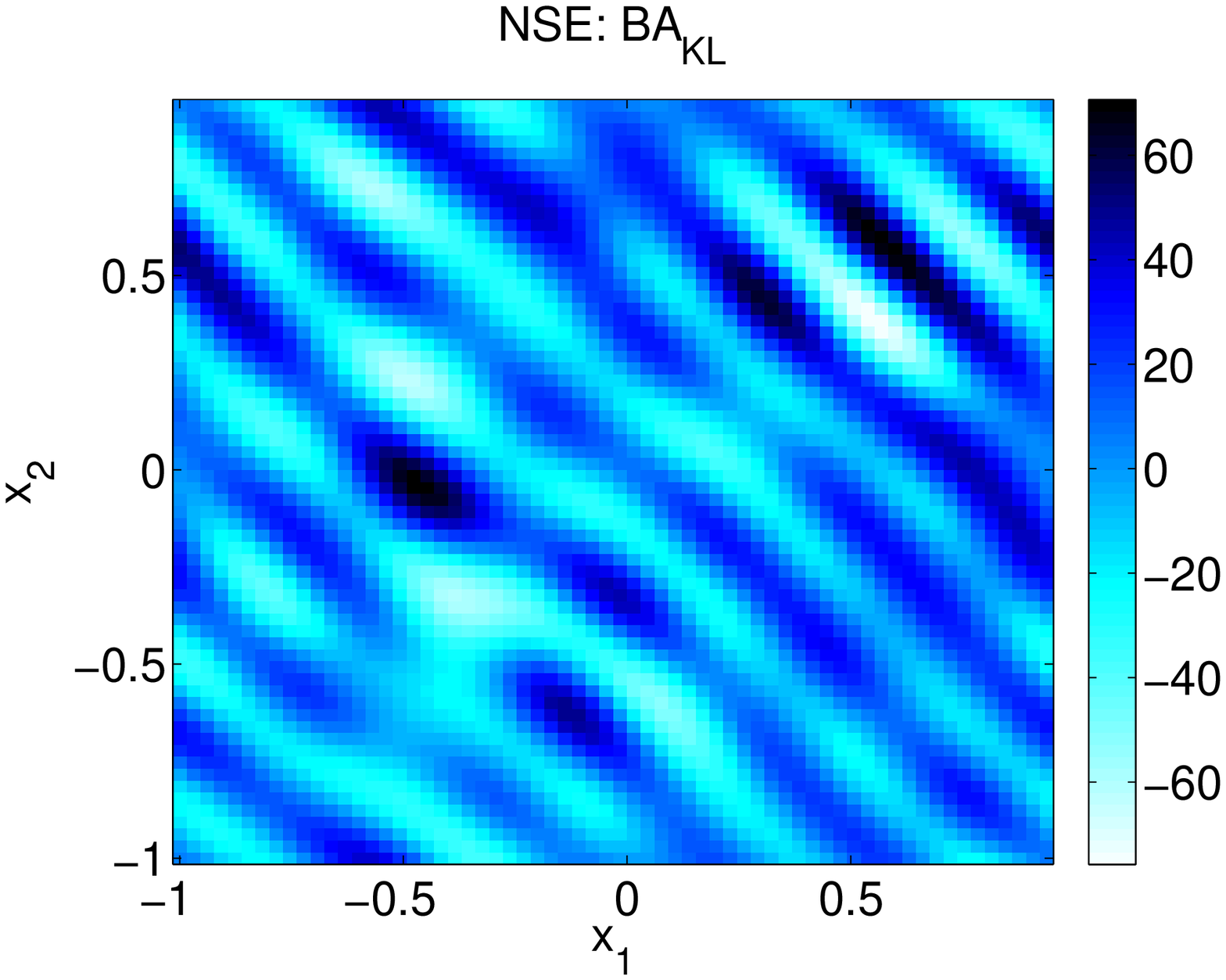}       
\includegraphics[width=0.24\textwidth]{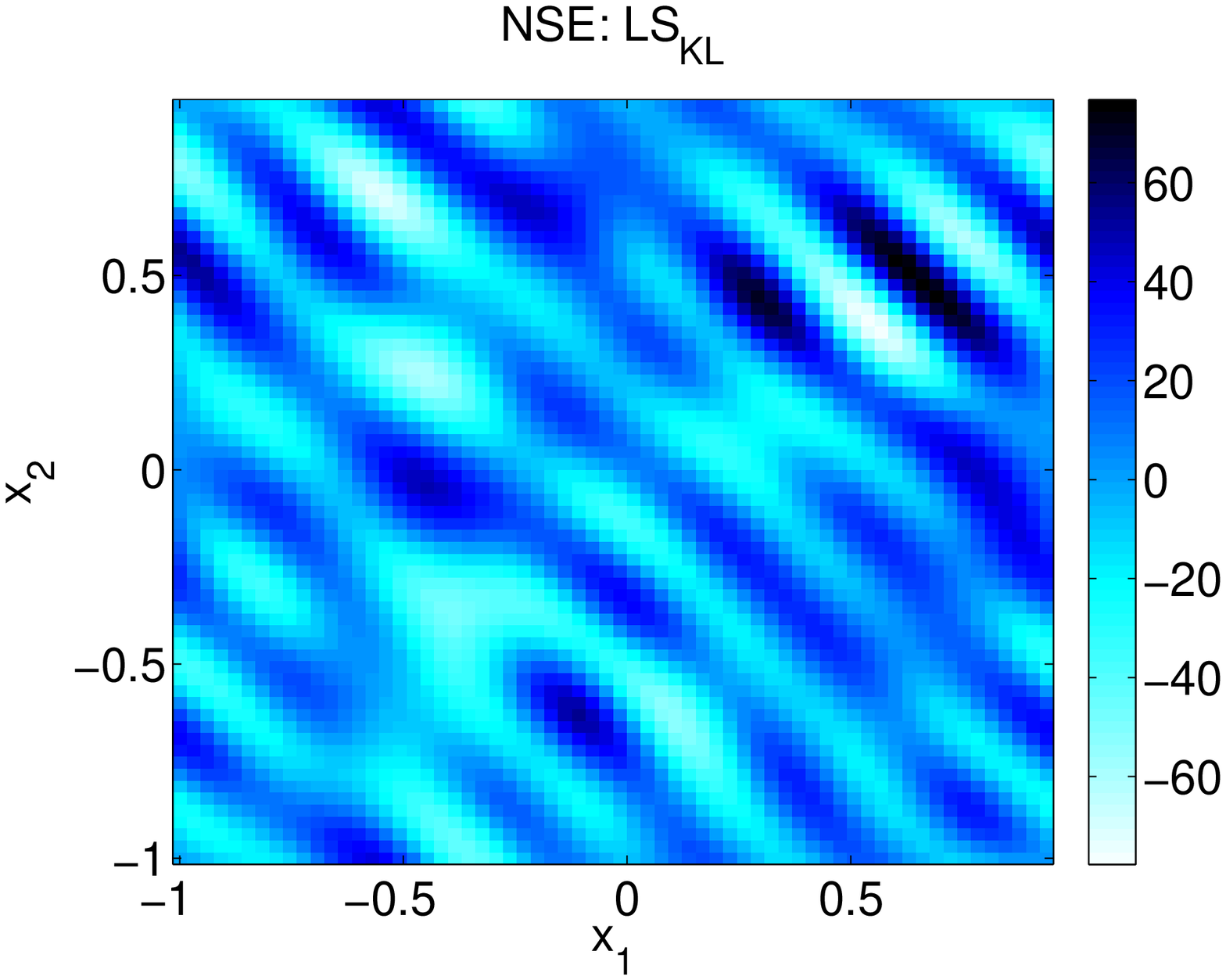}
       \includegraphics[width=0.24\textwidth]{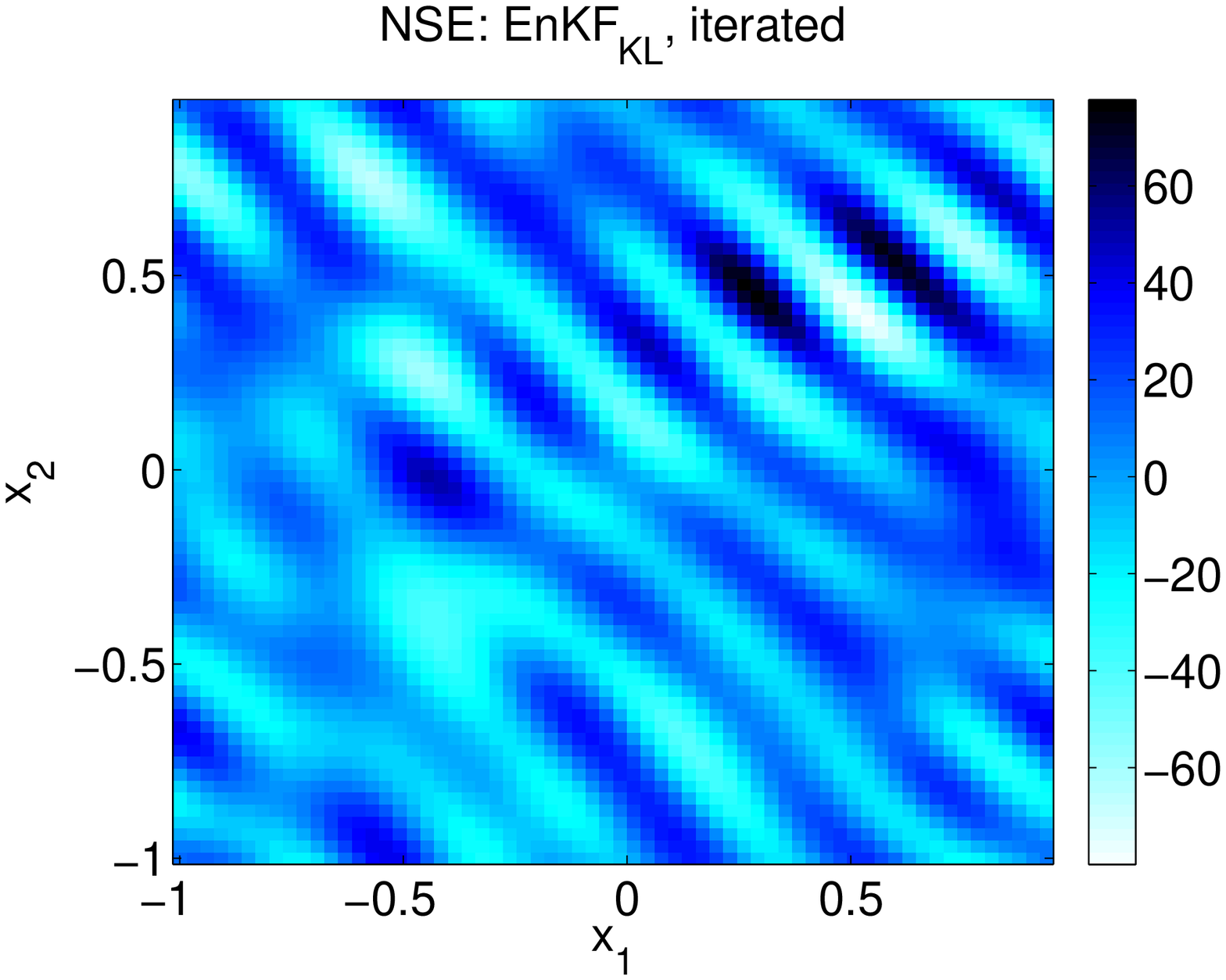} 
       \includegraphics[width=0.24\textwidth]{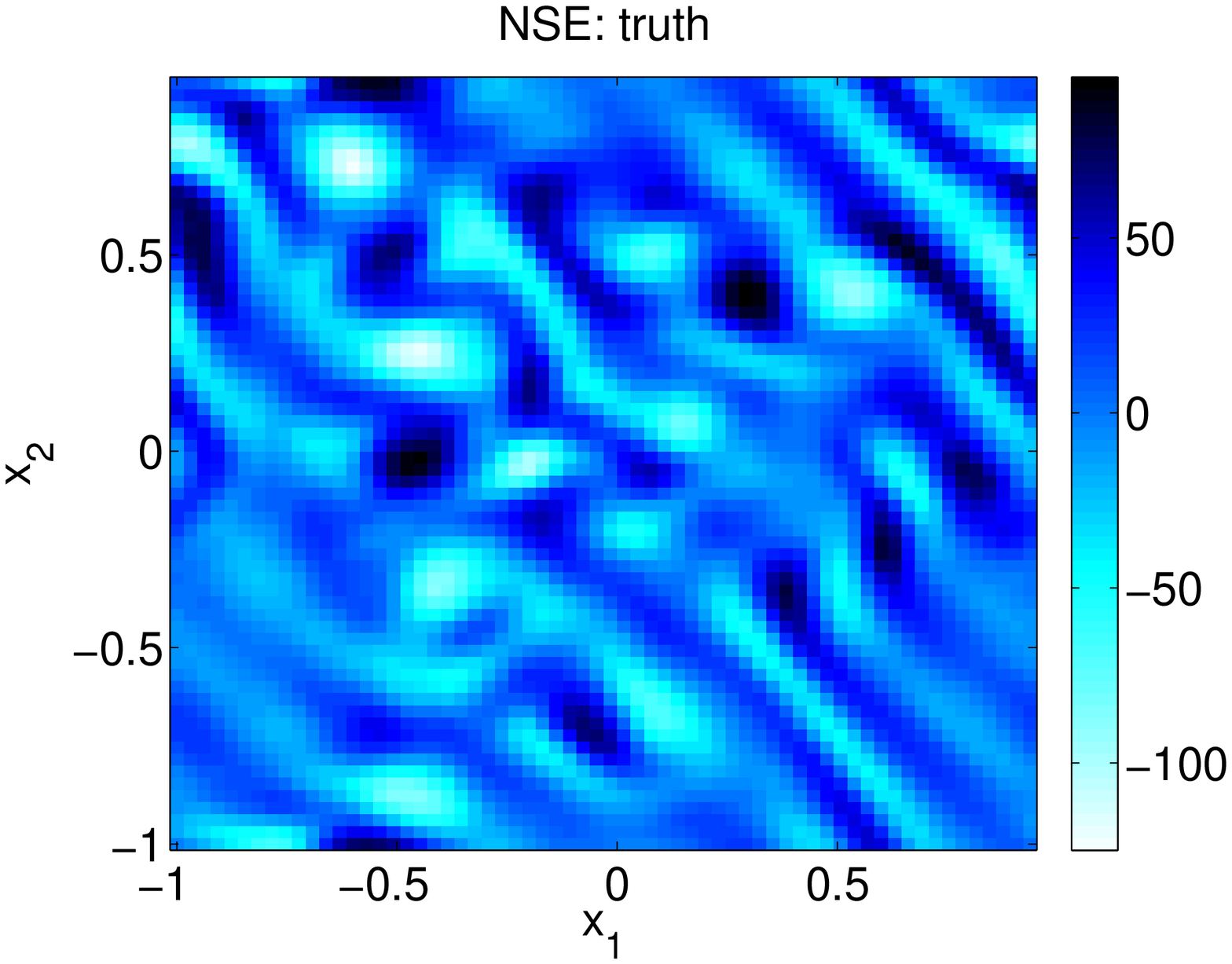}    
\caption{Vorticity of initial condition $\omega(0)$.
 Top left: BA$_{\rm R}$, average.
Top middle left: LS$_{\rm R}$, average.
Top middle right: first iteration of EnKF$_{\rm R}$, average.
Top right: tenth iteration of EnKF$_{\rm R}$, average.
Bottom left: BA$_{\rm KL}$. 
Bottom middle left: LS$_{\rm KL}$.  
Bottom middle right: EnKF$_{\rm KL}$.
Bottom right: the truth $u^\dagger$.}\label{Fign2}
\end{figure*}

\section{Conclusions}\label{Conclu}

We have illustrated the use of EnKF as a derivative-free
optimization tool for inverse problems, showing that
the method computes a nonlinear approximation in the
linear span of the initial ensemble. We have also demonstrated
comparable accuracy to least squares based methods
and shown that, furthermore, the accuracy is of
the same order of magnitude as the best approximation 
within the linear span of the initial ensemble.
Further study of the EnKF methodology for inverse problems,
and in particular its accuracy with respect to choice of
initial ensemble, would be of interest. Furthermore,
in this paper we have concentrated purely on the
accuracy of state estimation using EnKF. Study of
its accuracy in terms of uncertainty quantification will
yield further insight. Finally, although we have studied
a time-dependent example (the Navier-Stokes equation)
we did not use a methodology which exploited the
sequential acquisition of the data -- we concatenated
all the data in space and time. Exposing sequential structure in data can be useful in steady parameter estimation problems as shown in \cite{Calvetti}. In future work we will study ideas similar to those in our paper, but exploiting sequential structure.

\ack
This was supported by the ERC, EPSRC \& ONR.

\section*{References}
\bibliography{enkf_bib}
\bibliographystyle{plain}

\end{document}